\documentclass[12pt, reqno]{amsart}

\usepackage[margin=0.73in]{geometry}
\usepackage[
  bookmarks=true]{hyperref}
\usepackage[OT2, T1]{fontenc}
\usepackage[english]{babel}
\usepackage[utf8]{inputenc}
\usepackage{csquotes}
\usepackage[final]{microtype}
\usepackage{lmodern}
\usepackage{amsthm}
\usepackage{amssymb}
\usepackage{mathrsfs}
\usepackage{enumerate}
\usepackage{tikz-cd} 
\usepackage{tikz}
\usepackage{multicol}
\usepackage{multirow}
\usepackage{tikz-qtree}
\usepackage{rotating}
\usepackage{graphicx}
\usepackage{comment}
\usepackage{enumitem}
\usepackage{manfnt}
\usetikzlibrary{arrows,calc,matrix,trees,arrows.meta,positioning,decorations.pathreplacing,bending}

\newtheorem{theorem}{Theorem}[section]

\newtheorem{proposition}[theorem]{Proposition}
\newtheorem{lemma}[theorem]{Lemma}
\newtheorem{corollary}[theorem]{Corollary}

\theoremstyle{definition}

\newtheorem{remark}[theorem]{Remark}
\newtheorem{definition}[theorem]{Definition}
\newtheorem{example}[theorem]{Example}

\newtheorem{condition}[theorem]{Condition}

\newcommand{\Oh}{\mathcal{O}} 
\newcommand{\Sel}{\mathrm{Sel}} 
\newcommand{\Gal}{\mathrm{Gal}} 

\newcommand{\et}{\text{\'et}}
\newcommand{\Res}{\mathrm{Res}}

\newcommand{\genlegendre}[4]{%
  \genfrac{(}{)}{}{#1}{#3}{#4}%
  \if\relax\detokenize{#2}\relax\else_{\!#2}\fi
}

\DeclareSymbolFont{cyrletters}{OT2}{wncyr}{m}{n}
\DeclareMathSymbol{\Sha}{\mathalpha}{cyrletters}{"58}

\begin{document}

\title[Secondary terms for cyclic prime twist Selmer groups]{Secondary terms for first moments of Selmer groups of twists of elliptic curves over global function fields}
\author{Sun Woo Park}
\address{Max Planck Institute for Mathematics, Vivatsgasse 7, 53115 Bonn, Germany}
\email{\href{mailto:s.park@mpim-bonn.mpg.de}{s.park@mpim-bonn.mpg.de}}

\begin{abstract}
Let $E$ be a non-isotrivial elliptic curve over a global function field $\mathbb{F}_q(t)$ of characteristic coprime to $2$ and $3$. Under some explicit conditions, we determine the secondary terms for the first moments of prime Selmer groups of cyclic prime twist families of $E$ over $\mathbb{F}_q(t)$.
\end{abstract}

\maketitle
\tableofcontents

\section{Introduction}

There has been a wealth of research in understanding the moments and distributions of Selmer groups of families of abelian varieties over global fields. Relevant to the conjectural statements behind statistics of Selmer groups is a random matrix model formulated by Poonen and Rains \cite{PR12}, in which they show that these groups are intersections of two maximal isotropic subspaces lying in infinite dimensional vector spaces. There are a wide class of families of abelian varieties over which the heuristic by Poonen and Rains hold. These include the first moments of 2, 3, 4, and 5-Selmer groups of the universal family of elliptic curves (see for example \cite{BS15}), the distribution of $2^\infty$ Selmer groups of quadratic twist families of elliptic curves (see for example \cite{Sm22_01} and \cite{Sm22_02}), and distribution of odd prime Selmer groups of quadratic twist families of abelian varieties over global function fields (see for example \cite{EL23} and \cite{LL25}). 

A refinement of the problem of obtaining the moments and distributions of Selmer groups is the problem of determining rate of converge or secondary terms for the moments and distributions of Selmer groups. For quadratic twist families of abelian varieties, the monumental results by Alex Smith \cite{Sm22_01, Sm22_02}, Jordan Ellenberg and Aaron Landesman \cite{EL23}, and Aaron Landesman and Ishan Levy \cite{LL25} establish upper bounds for rates of convergence of statistics of their Selmer groups. The work by Alex Smith establishes unconditional results on essentially double logarithmic rate of convergence (with respect to the upper bound on absolute value of square-free elements twisting such abelian varieties) for the moments of their $2^\infty$ Selmer groups. The work by Aaron Landesman and Ishan Levy in particular uses stable homology results for Hurwitz modules to establish power saving error term (with respect to the degree of square-free polynomials twisting the given abelian variety) for the moments of their odd Selmer groups, under the assumption that the size of the constant field of the global function field is sufficiently large. For cyclic prime twist families of elliptic curves, a previous result by the author \cite{Pa22} establishes logarithmic saving error term for rates of convergence of distribution of their prime Selmer groups by using geometric ergodicity of irreducible aperiodic Markov operators. 

For universal families of elliptic curves, the recent breakthrough by Arul Shankar and Takashi Taniguchi \cite{ST25} proves the existence of the secondary term (also a power saving term with respect to the naive height) for the first moments of their $2$-Selmer groups. Their work shares its philosophy with a series of works by Takuro Shintani \cite{Shintani72, Shintani75}, in which he used properties of zeta functions associated to prehomogeneous vector spaces \cite{SS74} to compute the first moment of class numbers of irreducible integral binary cubic forms. Shinani's works, however, cannot be directly applied to the problem of computing secondary terms for the first moments of 2-Selmer groups. This is because the work by Shankar and Taniguchi crucially uses correspondences between 2-Selmer groups of elliptic curves and equivalence classes of integral binary quartic forms, the latter of which is not a prehomogeneous vector space. Instead, they use geometry of numbers method as in \cite{BST13}, devise novel approximation techniques of local functions, and perform Fourier analysis over equivalence classes of mod $n$ binary quartic forms (We sincerely thank Takashi Taniguchi for explaining the relation of \cite{ST25} with other previous relevant works in private communication).

The two results \cite{LL25, ST25} seem to suggest that over suitable families of elliptic curves (or abelian varieties) the secondary terms for moments of their Selmer groups are likely to be (or are) power saving terms in comparison to the main term. On the contrary, we construct twist families of elliptic curves where the secondary term for first moments of their Selmer groups are logarithmic - not power - saving term with respect to the height function. For example, let $F_n(\mathbb{F}_q)$ be the set of monic square-free polynomials over $\mathbb{F}_q$ of degree $n$. Given $f \in F_n(\mathbb{F}_q)$, a non-isotrivial elliptic curve $E$ over $\mathbb{F}_q(t)$ satisfying some mild conditions, and a small enough $\epsilon > 0$, we show that there exists an explicitly computable constant $\beta_E$ depending on $E$ such that the first moment of $2$-Selmer groups of quadratic twist families of $E$ satisfies the following equation for sufficiently large $n$ and $q$. 
\begin{equation}
    \sum_{f \in F_n(\mathbb{F}_q)} \# \mathrm{Sel}_2(E_f/\mathbb{F}_q(t)) = 3 \cdot (q^n-q^{n-1}) + \beta_E \cdot \frac{q^n - q^{n-1}}{n^{\frac{3}{8}}} + O_{E,\epsilon} \left( \frac{q^n - q^{n-1}}{n^{\frac{3}{8}} (\log n)^{\frac{1}{2}-\epsilon}} \right).
\end{equation}

\subsection{Setup}

To state the main result of the paper, we set up some notations and definitions. Let $K = \mathbb{F}_q(t)$ be the global function field over the finite field $\mathbb{F}_q$ of characteristic coprime to $2$ and $3$. Fix a prime number $l$ that is coprime to $\text{Char}(K)$. Let $E$ be an elliptic curve over $K$. Throughout this manuscript, we will assume the following four conditions.
\begin{condition} \label{condition}
Assume the following conditions on $K$, $l$, and $E$.
\begin{enumerate}
    \item The primitive $l$-th roots of unity is contained in $K$, i.e. $\mu_l \subset K$.
    \item The elliptic curve $E/K$ is non-isotrivial. 
    \item The elliptic curve $E/K$ admits a place of split multiplicative reduction.
    \item The Galois group $\Gal(K(E[l])/K)$ obtained from adjoining $l$-torsion points of $E/K$ contains the special linear group $\text{SL}_2(\mathbb{F}_l)$.
\end{enumerate}
\end{condition}
We denote by $F_{n}$ an open subscheme of $\mathbb{A}^n$ whose $\mathbb{F}_q$ rational points, denoted as $F_{n}(\mathbb{F}_q)$, parametrize monic square-free polynomials over $\mathbb{F}_q$ of degree $n$ which are coprime to the discriminant of $E$. 

Given a polynomial $f \in F_{n}(\mathbb{F}_q)$, we denote by $A_f$ the $(l-1)$ dimensional abelian variety over $K$ constructed as
\begin{equation}
    A_f / K := \text{Ker} \left(\text{Res}_{K}^{K(\sqrt[l]{f})} E \to E \right),
\end{equation}
where the projection map is induced from the field trace map with respect to the Galois extension $K(\sqrt[l]{f})/K$. We recall that the endomorphism ring of $A_f$ contains the cyclic generator $\sigma_{l,f}$ of $\Gal(K(\sqrt[l]{f})/K)$, and that there exists a $\Gal(\overline{K}/K)$-equivariant isomorphism
\begin{equation*}
    A_f[1-\sigma_{l,f}] \cong E[l],
\end{equation*}
see \cite[Section 4]{MR07} for further details. 

Denote by $\Sel_{1-\sigma_{l,f}} (A_f/K)$ the $1-\sigma_{l,f}$ Selmer group of $A_f/K$ obtained from the global and local long exact sequence of \'etale cohomology groups with respect to the short exact sequence of the following group scheme:
\begin{equation}
    0 \to A_f[1-\sigma_{l,f}] \to A_f \to A_f \to 0.
\end{equation}
We note that the Galois-equivariant isomorphism $A_f[1-\sigma_{l,f}] \cong E[l]$ implies that $\Sel_{1-\sigma_{l,f}} (A_f/K)$ is a finite dimensional $\mathbb{F}_l$ vector space which is contained in $H^1_{\et}(K, E[l])$.

\subsection{Main results}

In the previous work of the author \cite{Pa22}, we established that the distribution of dimensions of $\mathrm{Sel}_{1-\sigma_{l,f}}(A_f/K)$ converges to the Poonen-Rains distribution \cite{PR12} with explicit upper bounds on the rate of convergence as $f$ varies over the set of monic degree $n$ polynomials over $\mathbb{F}_q$. We generalize this result by computing the secondary term for the first moment of sizes of $\mathrm{Sel}_{1-\sigma_{l,f}}(A_f/K)$ as $f$ varies over the set of monic degree $n$ square-free polynomials over $\mathbb{F}_q$, but under the assumption that $n$ and $q$ are sufficiently large.

\begin{theorem}[Simplification of Theorem \ref{thm:main}] \label{introthm:main}
    Fix a prime number $l$. Let $E$ be an elliptic curve satisfying Condition \ref{condition}, and if $l = 2, 3, 5$ for some non-negative integer $m$, then $E$ satisfies an additional (yet explicitly computable) condition (see equation (\ref{eq:main_thm_0})). Then for sufficiently large $n$ and $q$ (see equation (\ref{eq:main_thm_1}) for explicit lower bounds for $n$ and $q$) and any small enough $\epsilon > 0$, there exists an explicit constant $\beta_{E,l}$ depending on $E$ and $l$ (see equation (\ref{eq:main_thm_2}) for its definition) such that
    \begin{equation}
        \sum_{f \in F_n(\mathbb{F}_q)} \# \mathrm{Sel}_{1-\sigma_{l,f}}(A_f/K) = (l+1) \cdot (q^n-q^{n-1}) + \beta_{E,l} \cdot \frac{q^n - q^{n-1}}{n^{\frac{l^2 - l + 1}{l^3}}} + O_{E,l,\epsilon}\left( \frac{q^n - q^{n-1}}{n^{\frac{l^2-l+1}{l^3}} (\log n)^{\frac{1}{2}-\epsilon}} \right),
    \end{equation}
    where the implied constant of the error term depends on $E$, $l$, and $\epsilon$, and is independent of $n$ and $q$.
\end{theorem}

\subsection{Strategy of the Proof}

Similar to the previous work of the author and a previous work by Klagsbrun, Mazur, and Rubin \cite{KMR14, Pa22}, we construct a governing Markov operator $M$ (see Definition \ref{defn:Markov}) over a countable state space $\mathbb{Z}_{\geq 0}$ that models the variation of dimensions of prime Selmer groups of cyclic twist families of elliptic curves over global function fields. The method we use to obtain secondary terms for the first moments of Selmer groups relies on using an explicit linear recursive formula that relates the first exponential moment of a probability distribution $\delta: \mathbb{Z}_{\geq 0} \to [0,1]$ with that of the distribution $M\delta$ obtained after applying the Markov operator $M$ to the distribution $\delta$. In particular, the following equation holds (c.f. equation (\ref{eqn:important2})):
\begin{equation*}
    \sum_{z=0}^\infty l^{z} \cdot \mathbb{P}[M\delta = z] = \left( \sum_{z=0}^\infty l^z \cdot \mathbb{P}[\delta = z] \right) \cdot \left(1 - \frac{l^2 - l + 1}{l^3} \right) + \left(1 + \frac{1}{l^3} \right).
\end{equation*}
Based on insights from \cite{KMR14, Pa22}, the number of iterations of $M$ we apply to some probability distribution $\delta$ to model the distribution of prime Selmer groups is given by the number of distinct irreducible factors of a square-free polynomial $f$ using which we cyclically twist the elliptic curve $E$. This number is asymptotically equal to $\log n$ as $n$ grows arbitrarily large (see for example Theorem \ref{thm:SatheSelberg}). Hence, one expects that the secondary term should be of order a fractional power of reciprocal of $n$ (which is roughly the size of $\log \# F_n(\mathbb{F}_q)$), and its fractional exponent should be closely related to the slope $\left(1 - \frac{l^2 - l + 1}{l^3} \right)$ of the linear recurrence relation.

To ensure that the secondary term can be obtained in align with our intuition, we construct a density $1$ subset of monic square-free polynomials of degree $n$, denoted as $F_n^{main,0}(\mathbb{F}_q)$ (see Definition \ref{def:key-subsets}), such that satisfies the following two conditions, under the assumption that $n$ and $q$ can be chosen to be sufficiently large:
\begin{itemize}
    \item The secondary term for the first moments of Selmer groups over the subset $F_n^{main,0}(\mathbb{F}_q)$ is determined by the slope of the linear recurrence relation. (Propositions \ref{prop:reinterpret-m0} and \ref{prop:convergence-m0})
    \item The rate of convergence of the first moments of Selmer groups over the complement of $F_n^{main,0}(\mathbb{F}_q)$ is of strictly smaller order than the secondary term obtained over the set $F_n^{main,0}(\mathbb{F}_q)$. (Propositions \ref{prop:Fn23-contribution}, \ref{prop:reinterpret-m1}, \ref{prop:reinterpret-[1]}, \ref{prop:convergence-m1}, and \ref{prop:convergence-[1]})
\end{itemize}
We construct this subset $F_n^{main,0}(\mathbb{F}_q)$ using these three criteria:
\begin{itemize}
    \item The number of distinct irreducible factors of $f \in F_n^{main,0}(\mathbb{F}_q)$ should not be too small (i.e. greater than $\frac{\log n}{\log \log \log n}$) or too large (i.e. smaller than a constant multiple of $\log n$).
    \item The number of distinct irreducible factors of $f \in F_n^{main,0}(\mathbb{F}_q)$ of small degree (say at most $4(\log n)^2$) should not be too large (i.e. smaller than  $\frac{\log n}{\log \log \log n}$).
    \item There is an irreducible factor of $f$ of large degree (say greater than $4(\log n)^2$) such that $E[l]$ has no non-trivial rational points over its residue field.
\end{itemize}
We note that a careful choice of $n$ and $q$ is required to extract the desired secondary term. For this we crucially use uniform versions of Sathe-Selberg theorem (Theorems \ref{thm:SatheSelberg}), large deviation estimates (Theorem \ref{thm:LDP}), and effective Chebotarev density theorem (Theorem \ref{thm:chebotarev}).

\section*{Acknowledgements}
Parts of the previous version of this manuscript appeared in Chapter 3 of the author's PhD thesis, but the content as well as the exposition of the manuscript has been significantly updated and corrected from what was written there. The majority of the manuscript was written during the author's stay at Max Planck Institute for Mathematics as a postdoctoral researcher, as well as his visit to AIM workshop on ``Nilpotent counting problems in arithmetic statistics''. 

We thank Jordan Ellenberg for giving patient guidence as well as insightful and enlightening comments, without which would not have led to these fruitful results. We thank Edgar Assing, Andrea Bianchi, Tim Browning, Valentin Blomer, Jordan Ellenberg, Anh Trong Nam Hoang, Daniel Keliher, Peter Koymans, Aaron Landesman, Ishan Levy, Wanlin Li, Will Sawin, Arul Shankar, Alex Smith, Takashi Taniguchi, Jiuya Wang, and Melanie Matchett-Wood for very helpful comments, and kindly and generously pointing out several inconsistencies and errors that were made in the process of understanding relations between Selmer spaces and Hurwitz spaces. We would like to sincerely thank the organizers of the AIM workshop on ``Nilpotent counting problems in arithmetic statistics'' - Brandon Alberts, Yuan Liu, and Melanie Matchett-Wood. Lastly, we would like to sincerely thank University of Wisconsin-Madison, Max Planck Institute for Mathematics, and American Institute of Mathematics for providing a wonderful environment to delve into research.

\section{Preliminary}

\subsection{Cyclic prime twists of elliptic curves}
\label{section:prime_twists}

This section is based on Mazur and Rubin's treatment of the cyclic prime twists of abelian varieties with respect to cyclic prime field extensions \cite[Chapter 3]{MR07}.

Let $l$ be a prime. Suppose $K$ includes all the primitive $l$-th roots of unity $\mu_l$. Given an element $f \in \Oh_K$, we denote by $\Res_K^{K(\sqrt[l]{f})} E$ the Weil restriction of scalars of the elliptic curve $E_K$ associated to the cyclic order-$l$ Galois extension $K(\sqrt[l]{f})/K$. As a group scheme over $K(\sqrt[l]{f})$, the Weil restriction of scalars of $E$ can be written as a product of $E$ with indices given by elements of the Galois group $\Gal(K(\sqrt[l]{f})/K)$:
\begin{equation}
    \Res_K^{K(\sqrt[l]{f})} E \cong \prod_{\tau \in \Gal(K(\sqrt[l]{f})/K)} E \; \; \; \; (\text{over } K(\sqrt[l]{f})).
\end{equation}
The Galois group $\Gal(K(\sqrt[l]{f}) / K)$ acts on the group scheme by cyclically permuting the summands indexed by the elements $\tau$ of the Galois group. This implies that the Weil restriction of scalars of $E$ has a canonical norm map to the elliptic curve $E$ defined over $K$:
\begin{equation}
    \mathbb{N}: \Res_K^{K(\sqrt[l]{f})} E \to E
\end{equation}
Using the norm map, we may decompose the semisimple group ring $\mathbb{Q}[\Gal(K(\sqrt[l]{f})/K)]$ as
\begin{equation}
    \mathbb{Q}[\Gal(K(\sqrt[l]{f})/K)] \cong \mathbb{Q} \oplus \mathbb{Q} \left[ \left( \Gal(K(\sqrt[l]{f})/K) \right)^\times \right]
\end{equation}
Denote by $\mathcal{I}_{f,l}$ the set
\begin{equation}
    \mathcal{I}_{f,l} = \mathbb{Q} \left[ \left( \Gal(K(\sqrt[l]{f})/K) \right)^\times \right] \cap \mathbb{Z}[\Gal(K(\sqrt[l]{f})/K)]
\end{equation}
so that $\mathcal{I}_{f,l}$ is an ideal of $\mathbb{Z}[\Gal(K(\sqrt[l]{f})/K)]$ as well as a $\Gal(\overline{K}/K)$-module.

\begin{definition} \label{definition:cyclic_prime_twist}
    Suppose $K$ contains all the primitive $l$-th roots of unity $\mu_l$. Let $f \in \Oh_K^\times / (\Oh_K^\times)^l$ be a $l$-th power free integral element over $K$. The cyclic $l$-twist of an elliptic curve $E$ associated to a cyclic extension $K(\sqrt[l]{f})/K$ is the following $l-1$ dimensional abelian variety over $K$:
    \begin{equation}
        \mathcal{I}_{f,l} \otimes E \cong \text{Ker} \left( \mathbb{N}: \Res_K^{K(\sqrt[l]{f})} E \to E \right).
    \end{equation}
\end{definition}
Throughout this manuscript, we may use the abbreviation $A_f$ to denote the $l-1$ dimensional abelian variety $\mathcal{I}_{f,l} \otimes E$ over $K$.

\begin{example}
Note that if $l = 2$, then the 1-dimensional abelian variety $\mathcal{I}_{f,2} \otimes E$ associated to a squarefree element $f \in \Oh_K^\times / (\Oh_K^\times)^2$ is isomorphic to the quadratic twist $E_f$ of $E$ over $K$. Suppose one has the Weierstrass model for the elliptic curve $E: y^2 = x^3 + Ax + B$. The Weil restriction of scalars with respect to the field extension $K(\sqrt{f})/K$ corresponds to the 2-dimensional abelian variety inside $\mathbb{A}^4 := \text{Spec}\left(K[x_0,x_1,y_0,y_1] \right)$ defined by the equation
\begin{equation}
    (y_0 + y_1 \sqrt{f})^2 = (x_0 + x_1 \sqrt{f})^3 + A (x_0 + x_1 \sqrt{f}) + B.
\end{equation}
Note that the $K$-rational coefficient terms for $\sqrt{f}$ and the $K$-rational terms cut out the 2-dimensional abelian variety inside $\mathbb{A}^4$:
\begin{align}
    \begin{split}
        y_0^2 + y_1^2 f &= x_0^3 + 3f x_0 x_1^2 + A x_0 + B \\
        2 y_0 y_1 &= 3 x_0^2 x_1 + f x_1^3 + A x_1
    \end{split}
\end{align}
If $y_0 = 0$, then one recovers the quadratic twist of the elliptic curve (assuming $x_1 = 0$)
\begin{equation}
    y_1^2 f = x_0^3 + Ax_0 + B.
\end{equation}
On the other hand, if $y_1 = 0$, then one recovers the elliptic curve (assuming $x_1 = 0$)
\begin{equation}
    y_0^2 = x_0^3 + Ax_0 + B.
\end{equation}
Because the action of $\Gal(K(\sqrt{f})/K)$ on $\Res_{K}^{K(\sqrt{f})} E$ is given by cyclic permutations of summands, it follows that as group schemes over $K(\sqrt{f})$,
\begin{equation}
    E \oplus E_f \cong \Res_{K}^{K(\sqrt{f})} E \cong E \oplus (\mathcal{I}_{f,2} \otimes E)
\end{equation}
which implies the desired isomorphism over $K$
\begin{equation}
    E_f \cong (\mathcal{I}_{f,2} \otimes E)
\end{equation}
because both group schemes are fixed by the action of $\Gal(K(\sqrt{f})/K)$. Throughout this manuscript, we will primarily use the notation $E_f$ to denote the quadratic twist of the elliptic curve $E$.
\end{example}

\begin{remark}
For any prime $l$, the short exact sequence of group schemes
\begin{equation}
    0 \to \mathcal{I}_{f,l} \otimes E \to \Res_{K}^{K(\sqrt[l]{f})} E \to E \to 0
\end{equation}
implies that
\begin{equation}
    \text{rank}_\mathbb{Z} (\mathcal{I}_{f,l} \otimes E)(K) = \text{rank}_{\mathbb{Z}} E(K(\sqrt[l]{f})) - \text{rank}_{\mathbb{Z}} E(K)
\end{equation}
\end{remark}

\begin{remark}
Recall that the Galois group $\Gal(K(\sqrt[l]{f}) / K)$ acts on $\Res_{K}^{K(\sqrt[l]{f})} E$ by cyclically permuting the summands indexed by the elements of the Galois group. This implies that the endomorphism ring of $\mathcal{I}_{f,l} \otimes E$ contains the group ring
\begin{equation}
    \text{End}(\mathcal{I}_{f,l} \otimes E) \supset \frac{\mathbb{Z} \left[ \left(\Gal(K(\sqrt[l]{f})/K)\right)^\times \right]}{(1 + \sigma_{l,f} + \cdots + \sigma_{l,f}^{l-1})}.
\end{equation}
where $\sigma_{l,f}$ is the generator of the Galois group $\Gal(K(\sqrt[l]{f})/K)$. With abuse of notation, we denote by $\sigma_{l,f}$ a fixed choice of a generator of the multiplicative group $\left(\Gal(K(\sqrt[l]{f})/K)\right)^\times$. We warn the readers that the generator $\sigma_{l,f}$ is not identical to the generator of the Galois group $\Gal(K(\mu_l)/K)$ for the case when $\mu_l \not\subset K$.
\end{remark}

\subsection{Selmer groups}

One of the well-studied strategies to bound the rank of an abelian variety $A$ is to construct its Selmer group associated to a choice of an element $m \in \text{End}(A)$. We recall the construction of Selmer groups from the following definition.
\begin{definition} \label{definition:m_Selmer}
Let $A_K$ be an abelian variety defined over a global field $K$. Suppose that $m \in \text{End}(A_K)$ has degree coprime to the characteristic of $K$. The $m$-Selmer group of the abelian variety $A_K$ is defined as a finite subspace of the first \'etale cohomology groups
\begin{equation} \label{defn:selmer}
    \Sel_m(A_K) := \text{Ker} \left( H^1_{\et}(K, A_K[m]) \to \prod_v H^1_{\et}(K_v, A_K)[m] \right),
\end{equation}
where the product over local cohomology groups spans over all finite places of $\Oh_K$. We note that the $m$-Selmer groups are constructed from comparing the long exact sequence of global and local \'etale cohomology groups with respect to the following short exact sequence of group schemes:
\begin{equation}
    0 \to A_K[m] \to A_K \to^m A_K \to 0.
\end{equation}
\end{definition}
We achieve the following short exact sequence
\begin{equation}
    0 \to A_K(K) / mA_K(K) \to \Sel_m(A_K) \to \Sha(A_K)[m] \to 0
\end{equation}
where $\Sha(A_K)[m]$ is the $m$-torsion subgroup of the Tate Shafarevich group $\Sha(A_K)$. Assuming that $A_K[m](K)$ is a finite $R$-module for some commutative finite ring $R$, we can obtain the upper bound of the rank of $A_K$ as
\begin{equation}
    \text{rank}_\mathbb{Z} A_K(K) + \text{rank}_R A_K[m](K) \leq \text{rank}_R \Sel_m(A_K)
\end{equation}
where the notation $\text{rank}_R M$ for a finitely generated $R$-module $M$ denotes
\begin{equation}
    \text{rank}_R M := \max_{\mathfrak{p} \subset R \text{ prime }} \left[ \dim_{R/\mathfrak{p}} (M \otimes_R R/\mathfrak{p}) \right].
\end{equation}

Using the structure of the endomorphism ring of cyclic twists of elliptic curves, we can now define the following Selmer groups of cyclic prime twists of elliptic curves.

\begin{definition} \label{definition:prime_Selmer_Groups}
    Let $E$ be any elliptic curve over $K$ which contains a primitive $l$-th root of unity $\mu_l$. Fix an $l$-th power free element $f \in \Oh_K^\times / (\Oh_K^\times)^l$. Using Definition \ref{definition:m_Selmer}, we can define the following two types of prime Selmer groups of abelian varieties $\mathcal{I}_{f,l} \otimes E$.
    \begin{enumerate}
    \item The $p$-Selmer group of $\mathcal{I}_{f,l} \otimes E$ is a finite dimensional $\mathbb{F}_p$-vector space
    \begin{equation}
        \Sel_p(\mathcal{I}_{f,l} \otimes E) \subset H^1_{\et}(K, (\mathcal{I}_{f,l} \otimes E)[p]).
    \end{equation}
    \item Recall that $\sigma_{l,f}$ is a generator of the multiplicative group $\Gal(K(\sqrt[l]{f})/K)^\times$. The $1-\sigma_{l,f}$ Selmer group of $\mathcal{I}_{f,l} \otimes E$ is a finite dimensional $\mathbb{F}_l$-vector space
    \begin{equation}
        \Sel_{1-\sigma_{l,f}}(\mathcal{I}_{f,l} \otimes E) \subset H^1_{\et}(K,(\mathcal{I}_{f,l} \otimes E)[1 - \sigma_{l,f}]) \cong H^1_{\et}(K,E[l])
    \end{equation}
    where the last isomorphism follows from the canonical $\Gal(\overline{K}/K)$-module isomorphism 
    \begin{equation}
        (\mathcal{I}_{f,l} \otimes E)[1 - \sigma_{l,f}] \cong E[l].
    \end{equation}
    (See \cite{MR07}[Proposition 4.1] for a complete proof of this fact).
    \end{enumerate}
\end{definition}

\subsection{Subsets of squarefree polynomials}

We borrow some (but not all) of the notations from \cite{Pa22} which allow us to subdivide the set of squarefree polynomials $F_n(\mathbb{F}_q)$. 
\begin{definition}
    The following notations are adapted from \cite[Definition 4.2, 4.6, 4.10]{Pa22}, which adhere to the notations as defined in \cite[Section 3,4,5]{KMR14}. Throughout these notations, we denote by $E$ a non-isotrivial elliptic curve over $K = \mathbb{F}_q(t)$. We also define $\mathfrak{n} := \lfloor 4(\log n)^2 \rfloor$.
    \begin{itemize}
    \item $\Sigma_E$: A set of places of $K$ consisting only of places of bad reduction of $E$.
    \item $\Sigma_E(\sigma)$: A set of places of $K$ that includes a set of places in $\Sigma_E$ and a set of places dividing a square-free element $\sigma \in K$.
    \item For $0 \leq i \leq 2$, define the set
    \begin{equation*}
        \mathcal{P}_i := \{v \text{ place of } K \; | \; v \not\in \Sigma_E \text{ and } \dim_{\mathbb{F}_p} E(K_v)[l] = i \}
    \end{equation*}
    Denote by $\mathcal{P}$ the set $\cup_i \mathcal{P}_i$.
    \item Given a positive integer $d \in \mathbb{N}$, the set $\mathcal{P}_i(d)$ and $\mathcal{P}(d)$ is the set of irreducible polynomials in respective sets of degree equal to $d$.
    \item $f$: A square-free polynomial of degree $n$ over $\mathbb{F}_q$ that is coprime to the any places in $\Sigma_E$.
    \item $f_*$: A product of irreducible factors of $f$ whose degrees are at most $\mathfrak{n}$, i.e.
    \begin{equation*}
    f_* = \prod_{\substack{g \mid f \\ g \in \cup_{i=1}^{\mathfrak{n}} \mathcal{P}(i)}} g.
    \end{equation*}
    \item $f^*$: A product of irreducible factors of $f$ whose degrees are at least $\mathfrak{n}+1$, i.e. 
    \begin{equation*}
    f^* = \prod_{\substack{g \mid f \\ g \in \cup_{i=\mathfrak{n}+1}^{n} \mathcal{P}(i)}} g.
    \end{equation*}
    \item $\omega(f)$: The number of irreducible factors of a square-free polynomial $f$ over $\mathbb{F}_q$.
    \end{itemize}
\end{definition}

We will subdivide $F_n(\mathbb{F}_q)$ based on considering these three criteria.
\begin{itemize}
    \item What is the cardinality of $\omega(f)$?
    \item What is the cardinality of $\omega(f_*)$?
    \item How many irreducible factors of $f$ lie in $\mathcal{P}_0$?
\end{itemize}
\begin{definition} \label{def:key-subsets}
    Let $R_1, R_2$ be positive integers. We define four subsets of $F_n(\mathbb{F}_q)$ as follows.
    \begin{itemize}
        \item $F_{n}^{main}(\mathbb{F}_q) := \Big\{ f \in F_n(\mathbb{F}_q) : R_2 < \omega(f) \leq R_1, \omega(f_*) \leq R_2, f \text{ has an irred. factor in } \mathcal{P}_0 \cup \mathcal{P}_1 \Big\},$
        \item $F_n^{[1]}(\mathbb{F}_q) := \Big\{ f \in F_n(\mathbb{F}_q) : R_2 < \omega(f) \leq R_1, \omega(f_*) \leq R_2, \text{ all irred. factors of } f \text{ are in } \mathcal{P}_2 \Big\},$
        \item $F_n^{[2]}(\mathbb{F}_q) := \Big\{ f \in F_n(\mathbb{F}_q) : R_2 < \omega(f) \leq R_1, \omega(f_*) > R_2 \Big\}$,
        \item $F_n^{[3]}(\mathbb{F}_q) := \Big\{ f \in F_n(\mathbb{F}_q) : \omega(f) > R_1 \text{ or } \omega(f) \leq R_2 \}$.
    \end{itemize}
    By construction we have $F_n(\mathbb{F}_q) = F_n^{main}(\mathbb{F}_q) \bigsqcup F_n^{[1]}(\mathbb{F}_q) \bigsqcup F_n^{[2]}(\mathbb{F}_q) \bigsqcup F_n^{[3]}(\mathbb{F}_q)$.
\end{definition}
As the name suggests, we shall choose $R_1$ and $R_2$ (which will depend on $n$) such that the distribution of twist families of Selmer groups over the set $F_n^{main}(\mathbb{F}_q)$ determines the distribution of those over $F_n(\mathbb{F}_q)$, whereas the distribution over $F_n^{[1]}(\mathbb{F}_q)$, $F_n^{[2]}(\mathbb{F}_q)$, and $F_n^{[3]}(\mathbb{F}_q)$ are regarded as error terms. To determine the lower bound of the rate of convergence of the distribution, we will also need to demonstrate that the contributions from the subsets $F_n^{[1]}(\mathbb{F}_q)$, $F_n^{[2]}(\mathbb{F}_q)$, and $F_n^{[3]}(\mathbb{F}_q)$ are less than the rate of convergence of the distribution over $F_n^{main}(\mathbb{F}_q)$.

To compute the rate of convergence of the distribution over the subset $F_n^{main}(\mathbb{F}_q)$, we will use the function field analogue of Sathe-Selberg theorem.
\begin{theorem}[Theorem 1 of \cite{AP19}] \label{thm:SatheSelberg}
    Let $A > 1$ be a fixed constant. Then uniformly for all $n \geq 2$ and $1 \leq k \leq A \log n$,
    \begin{equation}
        \# \{f \in F_n(\mathbb{F}_q) : \omega(f) = k\} = \frac{q^n}{n} \frac{(\log n)^{k-1}}{(k-1)!} \cdot \Big( G\Big( \frac{k-1}{\log n} \Big) + O_A(k/(\log n)^2) \Big),
    \end{equation}
    where the function $G$ is defined as
    \begin{equation*}
        G(z) := \frac{1}{\Gamma(1 + z)} \cdot \prod_{g \in \mathcal{P}} \Big(1 + \frac{z}{q^{\deg g}} \Big) \Big( 1 - \frac{1}{q^{\deg g}} \Big)^z. 
    \end{equation*}
\end{theorem}

To help us with understanding the contributions from the subset $F_n^{[3]}(\mathbb{F}_q)$, we recall the following upper bound on the number of squarefree polynomials of degree $n$ having exactly $k$ factors.
\begin{theorem}[Theorem 5.4 of \cite{GKMZ20}] \label{thm:LDP}
    We have the uniform upper bound for all $k$ and $n$:
    \begin{equation}
        \# \{f \in F_n(\mathbb{F}_q) : \omega(f) = k\} \leq \frac{q^n}{n} \frac{(\log n + 2 - \log 2)^{k-1}}{(k-1)!}.
    \end{equation}
\end{theorem}

The following result will be useful for understanding the contributions from the subset $F_n^{[2]}(\mathbb{F}_q)$.
\begin{theorem}[Proposition 4.13 of \cite{Pa22}] \label{thm:exceptional_set}
    Suppose $n > \exp(\exp(\exp(\exp(e))))$. If $R_2 = \lfloor \frac{\log n}{\log \log \log n} \rfloor$, then we have
    \begin{equation}
        \frac{\# F_n^{[2]}(\mathbb{F}_q)}{\# F_n(\mathbb{F}_q)} < 4 \cdot n^{-(\log \log n)^{1 - (\log \log \log n)^{-1/2}}} < 4 n^{-77762}.
    \end{equation}
\end{theorem}
\begin{proof}
    The result follows from adapting the proof of \cite[Proposition 4.13]{Pa22} where we substitute $m_{n,q} := \log n + \log \log q$ with $\log n$. The exponent $-77762$ comes from a crude estimate of the function $-(\log \log n)^{1 - (\log \log \log n)^{-1/2}}$ when $n = \exp(\exp(\exp(\exp(e))))$.
\end{proof}

Lastly, to obtain the contributions from the subset $F_n^{[1]}(\mathbb{F}_q)$, as well as constructing the governing stochastic process for computing distribution of twist families of Selmer groups over $F_n^{main}(\mathbb{F}_q)$, we use the effective Chebotarev density theorem.
\begin{theorem}[Proposition 6.4.8 of \cite{FJ08} and Theorem 3.1 of \cite{Pa22}] \label{thm:chebotarev}
    Given a geometric Galois extension $L/K$, let $C \subset G = \mathrm{Gal}(L/K)$ be a conjugacy class. Denote by $g_L$ the genus of $L$. Then
    \begin{equation}
        \Big|\#\Big\{ v \text{ a place of } K : \mathrm{Frob}_v \in C, \dim_{\mathbb{F}_q} ({\mathcal{O}_K}/{(v)}) = n \Big\} - \frac{|C|}{|G|} \frac{q^n}{n} \Big| < \frac{6|C|}{n|G|} (|G| + g_L) q^{n/2}.
    \end{equation}
\end{theorem}

\subsection{Markov operators governing distributions of Selmer groups}
\label{section:markov}

This section is based on the author's previous work on generalizing the work by Klagsbrun, Mazur, and Rubin \cite{KMR14} to the global function field setting. We borrow some (but not all) of the notations from \cite{Pa22} which will be used to compute the lower bound on the range of convergence of the average size of the $1-\sigma_{l,f}$ Selmer groups of the abelian variety $A_f$ over $K$ as $f$ ranges over square-free polynomials of degree $n$ over $\mathbb{F}_q$.

\begin{definition}
    The following notations are adapted from \cite[Definition 5.1, 5.2]{Pa22}, which adhere to the notations as defined in \cite[Section 5]{KMR14}.
    \begin{itemize}
    \item $\Omega_\sigma$: The set of finite Cartesian products of local character
    \begin{equation*}
        \chi := (\chi_v)_v \in \Omega_\sigma := \prod_{v \in \Sigma_E(\sigma)} \text{Hom} ( \text{Gal}(\overline{K}_v/K_v), \mu_l)
    \end{equation*}
    such that the component $\chi_v$ is ramified if $v \mid \sigma$.
    \item $\Omega_E$: The set of finite Cartesian products of local characters 
    \begin{equation*}
        \chi := (\chi_v)_v \in \Omega_\sigma := \prod_{v \in \Sigma_E} \text{Hom} ( \text{Gal}(\overline{K}_v/K_v), \mu_l).
    \end{equation*}
    \item $\Omega_{\chi,\mathfrak{v}}$: Given a character $\chi \in \Omega_\sigma$, it is a set of characters $\chi' \in \Omega_{\sigma \mathfrak{v}}$ such that \textbf{(1)} for any $v \in \Sigma_E(\sigma)$, $\chi'_v = \chi_v$, and \textbf{(2)} at $\mathfrak{v}$, $\chi'_\mathfrak{v}$ is ramified. Note that $\# \Omega_{\chi, \mathfrak{v}} = l(l-1)$.
    \item Given $\chi_v \in \text{Hom}( \text{Gal}(\overline{K}_v/K_v), \mu_l)$, we denote by $A_{\chi_v}$ the $l-1$ dimensional abelian variety over $K_v$ defined as
    \begin{equation*}
        A_{\chi_v} := \text{Ker}\left( \text{Res}_{K_v}^{K_v^{\chi_v}} E \to E \right),
    \end{equation*}
    where $K_v^{\chi_v}$ is the cyclic order-$l$ extension over $K_v$ induced from the cyclic character $\chi_v$.
    \item Given $\chi \in \Omega_\sigma$, the local Selmer group of $E$ associated to $\chi$ is defined as
    \begin{equation*}
        \Sel(E[l], \chi) := \text{Ker} \left( H^1_{\et}(K, E[l]) \to \prod_v H^1_{\et}(K_v,E[l]) / \mathcal{H}_v^\chi \right),
    \end{equation*}
    where
    \begin{equation*}
        \mathcal{H}_v^\chi := \begin{cases}
            &\text{im} \left( \delta_v^\chi: A_{\chi_v}(K_v) / \pi A_{\chi_v}(K_v) \to H^1(K_v,E[l]) \right) \\
            &H^1(\mathcal{O}_{K_v},E[l]).
        \end{cases}
    \end{equation*}
    Given a square-free polynomial $f$, we choose a cyclic order-$l$ character $\chi_f \in \text{Hom}(\text{Gal}(\overline{K}/K), \mu_l)$ whose fixed field is the cyclic order-$l$ extension $K(\sqrt[l]{f})$. It follows that 
    \begin{equation*}
        \Sel_{1-\sigma_{l,f}}(A_f/K) = \Sel(E[l], (\chi_{f,v})_{v \in \Sigma_E(f)})
    \end{equation*}
    where $\chi_{f,v}$ is the restriction of the global character $\chi_f$ at a place $v$.
    \item Given $\chi \in \Omega_\sigma$, we denote by $\mathrm{rk}(\chi)$ to be
    \begin{equation*}
        \mathrm{rk}(\chi) := \dim_{\mathbb{F}_l} \mathrm{Sel}(E[l], \chi).
    \end{equation*}
    \end{itemize}
    \item Given $\chi \in \Omega_\sigma$ and a place $\mathfrak{v}$ of $K$, we denote by $t_\chi(\mathfrak{v})$ the image of $\mathrm{Sel}(E[l],\chi)$ with respect to the localization map at $\mathfrak{v}$, i.e.
    \begin{equation*}
        t_\chi(\mathfrak{v}) := \dim_{\mathbb{F}_l} \mathrm{Image} \Big( \mathrm{loc}_{\mathfrak{v}}: \mathrm{Sel}(E[l],\chi) \to H^1(\mathcal{O}_{K_v}, E[l]) \Big).
    \end{equation*}
\end{definition}

Using these notations, we recall some results from \cite{Pa22} which will be crucial to determine the rate of convergence of distribution of twist families of Selmer groups. 

The first result is \cite[Proposition 5.3]{Pa22}, which compares the dimension of local Selmer groups of $\mathrm{Sel}(E[l], \chi')$ and $\mathrm{Sel}(E[l],\chi)$ such that $\chi' \in \Omega_{\chi,\mathfrak{v}}$.
\begin{proposition}[Proposition 5.3 of \cite{Pa22}] \label{prop:Selmer_vary}
    Let $E$ be an elliptic curve which satisfies Condition \ref{condition}. Given a square-free product of places $\sigma$ coprime to the set of elements in $\Sigma_E$, let $\mathfrak{v}$ be a place such that $\mathfrak{v} \not\in \Sigma_{E}(\sigma)$. Then given a fixed character $\chi \in \Omega_{\sigma}$, we have the following equation for any $\chi' \in \Omega_{\chi, \mathfrak{v}}$.
    \begin{equation*}
        \mathrm{rk}(\chi') - \mathrm{rk}(\chi) = \begin{cases}
            2 &\text{ if } \mathfrak{v} \in \mathcal{P}_2 \text{ and } t_{\chi}(\mathfrak{v}) = 0 \text{ for exactly } l-1 \text{ many } \chi' \in \Omega_{\chi,\mathfrak{v}}, \\
            1 &\text{ if } \mathfrak{v} \in \mathcal{P}_1 \text{ and } t_{\chi}(\mathfrak{v}) = 0, \\
            -1 &\text{ if } \mathfrak{v} \in \mathcal{P}_1 \text{ and } t_{\chi}(\mathfrak{v}) = 1, \\
            -2 &\text{ if } \mathfrak{v} \in \mathcal{P}_2 \text{ and } t_{\chi}(\mathfrak{v}) = 2, \\
            0 &\text{ otherwise}.
        \end{cases}
    \end{equation*}
\end{proposition}

An immediate corollary of Proposition \ref{prop:Selmer_vary} is the following trivial upper bound on the dimension of Selmer groups, which will be useful to control the contributions from the sets $F_n^{[2]}(\mathbb{F}_q)$ and $F_n^{[3]}(\mathbb{F}_q)$.
\begin{corollary} \label{cor:Selmer_vary}
    Given a square-free polynomial $f$ of degree $m$, we have 
    \begin{equation*}
        \dim_{\mathbb{F}_l} \mathrm{Sel}_{1-\sigma_{l,f}}(A_f/K) \leq D_{E,l} + 2 \omega(f),
    \end{equation*}
    where $D_{E,l}$ is an explicitly computable constant depending only on the degree of $\Delta_E$ and $l$, and is independent of $n$ and $q$.
\end{corollary}
\begin{proof}
    By Proposition \ref{prop:Selmer_vary}, we have
    \begin{equation*}
        \dim_{\mathbb{F}_l} \mathrm{Sel}_{1-\sigma_{l,f}}(A_f/K) = \dim_{\mathbb{F}_l} \mathrm{Sel}(E[l], (\chi_{f,v})_{v \in \Sigma_E(f)}) \leq \max_{\chi \in \Omega_E} \mathrm{rk}(\chi) + 2 \omega(f).
    \end{equation*}
    Using Tate's algorithm \cite[Table 4.1, p.365]{AdvancedAEC} and upper bounds on Tamagawa factors at places of purely additive reduction \cite[Corollary 1.15]{LO85}, we obtain that there exists an explicitly computable constant $D_{E,l}$ (depending only on the degree of $\Delta_E$ and $l$) such that ${\max_{\chi \in \Omega_E} \mathrm{rk}(\chi)} \leq D_{E,l}$.
\end{proof}

Next we state two applications of Chebotarev density theorem obtained from \cite{Pa22}. The first application regards effective Chebotarev density theorem for $(\mathbb{Z}/l \mathbb{Z})^{\oplus m}$ Galois extensions over $K$.
\begin{proposition}[Corollary 4.17 of \cite{Pa22}] \label{prop:cheb1}
    Let $E$ be an elliptic curve which satisfies Condition \ref{condition}. Suppose $h_1, h_2, \cdots, h_w$ are irreducible polynomials over $\mathbb{F}_q$. Choose $n, q$ such that $\sum_{k=1}^w h_k \leq n$ and $w \leq \rho \log n$ for some $\rho > 1$.
    \begin{enumerate}
        \item Suppose $l \geq 5$, or $K(\sqrt[l]{h_1}, \cdots, \sqrt[l]{h_k}) \cap K(E[l]) = K$. Then for any element $a \in \mu_l^{\oplus w}$,
        \begin{equation}
            \Big| \frac{\#\{v \in \mathcal{P}_k(i) : \left( \left( \frac{v}{h_k} \right) \right)_{k=1}^w = a \in \mu_l^{\oplus w}\}}{\# \mathcal{P}_k(i)} - \frac{1}{p^{w}} \Big| < 6(2g_{E[l]} + 2l^3 - l - 2) \cdot l^w \cdot \frac{q^{-i/2}}{i}.
        \end{equation}
        In fact, if $i > \mathfrak{n}$, then
        \begin{equation}
            \Big| \frac{\#\{v \in \mathcal{P}_k(i) : \left( \left( \frac{v}{h_k} \right) \right)_{k=1}^w = a \in \mu_l^{\oplus w}\}}{\# \mathcal{P}_k(i)} - \frac{1}{p^{w}} \Big| < 6(2g_{E[l]} + 2l^3 - l - 2) \cdot n^{-2 \log n + \rho \log l}.
        \end{equation}
        \item Suppose $l \leq 3$, and $K(\sqrt[l]{h_1}, \cdots, \sqrt[l]{h_k}) \cap K(E[l]) \neq K$. Then there are $l^w - l^{w-1}$ many elements $a \in \mu_l^{\oplus w}$ such that $\left( \left( \frac{v}{h_k} \right) \right)_{k=1}^w \neq a$ for all $v \in \mathcal{P}_k(i)$. For the other $l^{w-1}$ many elements $a \in \mu_l^{\oplus w}$, we have
        \begin{equation}
            \Big| \frac{\#\{v \in \mathcal{P}_k(i) : \left( \left( \frac{v}{h_k} \right) \right)_{k=1}^w = a \in \mu_l^{\oplus w}\}}{\# \mathcal{P}_k(i)} - \frac{1}{p^{w-1}} \Big| < 6(2g_{E[l]} + 2l^3 - l - 2) \cdot l^w \cdot \frac{q^{-i/2}}{i}.
        \end{equation}
        In fact, if $i > \mathfrak{n}$, then
        \begin{equation}
            \Big| \frac{\#\{v \in \mathcal{P}_k(i) : \left( \left( \frac{v}{h_k} \right) \right)_{k=1}^w = a \in \mu_l^{\oplus w}\}}{\# \mathcal{P}_k(i)} - \frac{1}{p^{w-1}} \Big| < 6(2g_{E[l]} + 2l^3 - l - 2) \cdot n^{-2 \log n + \rho \log l}.
        \end{equation}
    \end{enumerate}
\end{proposition}
\begin{proof}
    The statements provided above are immediate corollaries of \cite[Corollary 4.17]{Pa22}. The explicit constant for the error term when $i > \mathfrak{n}$ is obtained in \cite[p.3294]{Pa22} below equation (57).
\end{proof}
The second application regards effective Chebotarev density theorem for determining $t_\chi(\mathfrak{v})$ for a given $\chi \in \Omega_\sigma$.
\begin{proposition}[Proposition 5.4 of \cite{Pa22}] \label{prop:cheb2}
    Let $E$ be an elliptic curve which satisfies Condition \ref{condition}. Fix a square-free product of places $\sigma$ coprime to $\Sigma_E$, and a local character $\chi \in \Omega_\sigma$. Let $d_{i,j}$ be quantities defined as follows.
    \begin{equation}
        d_{i,j} = \begin{cases}
        1 &\text{ if } (i,j) = (0,0), \\
        1 - l^{-\mathrm{rk}(\chi)} &\text{ if } (i,j) = (1,-1), \\
        l^{-\mathrm{rk}(\chi)} &\text{ if } (i,j) = (1,1), \\
        1 - (l+1)l^{-\mathrm{rk}(\chi)} + l^{1 - 2\mathrm{rk}(\chi)} &\text{ if } (i,j) = (2,-2), \\
        (l+1)(l^{-\mathrm{rk}(\chi)} - l^{-2\mathrm{rk}(\chi)}) &\text{ if } (i,j) = (2,0), \\
        l^{-2 \mathrm{rk}(\chi)} &\text{ if } (i,j) = (2,2), \\
        0 &\text{ otherwise}.
        \end{cases}
    \end{equation}
    Let $D_{E,l}$ be the constant from Corollary \ref{cor:Selmer_vary}. Then
    \begin{equation}
        \Big| \frac{\#\{\mathfrak{v} \in \mathcal{P}_i(d) : \mathfrak{v} \not\in \Sigma_E(\sigma), t_\chi(\mathfrak{v}) = j\}}{\#\{\mathfrak{v} \in \mathcal{P}_i(d) : \mathfrak{v} \not\in \Sigma_E(\sigma)\}} - d_{i,j} \Big| < 16 \cdot D_{E,l} \cdot l^{4 + 3 \# \Sigma_E(\sigma)} \cdot q^{-d/2}
    \end{equation}
    for every $d > (12 \log l + 2 \log D_{E,l} + (6 \log l) \cdot \# \Sigma_E) / \log q$. In particular, if $q > e^{12 \log l + 2 \log D_{E,l} + (6 \log l) \cdot \# \Sigma_E}$, then the above equation holds for every $d \geq 1$.
\end{proposition}
\begin{proof}
    The explicit constant for the error term is obtained in \cite[p.3299]{Pa22} below equation (71). By the proof of Corollary \ref{cor:Selmer_vary}, we have the upper bound $l^{\max_{\chi \in \Omega_E} \mathrm{rk}(\chi)} < D_{E,l}$.
\end{proof}

Combining the previous two applications of Chebotarev density theorem, the variations on dimensions of Selmer groups with respect to increment in number of irreducible factors dividing the conductor of the order $l$ cyclic character $\chi_f$ can be modeled by Markov operators. We recall the Lagrangian mod-$l$ Markov operator constructed from \cite{KMR14}, which governs the variations of dimensions of $1-\sigma_{l,f}$ Selmer groups of $A_f$ over $K$. We borrow the notations used in Section 6 of \cite{Pa22}.
\begin{definition}[Definition 6.1 of \cite{Pa22}] \label{defn:Markov}
    Let $M_L := [l_{r,s}]$ be the operator over the state space of non-negative integers $\mathbb{Z}_{\geq 0}$ defined as
    \begin{equation*}
        l_{r,s} := \begin{cases}
            1 - l^{-r} &\text{ if } s = r-1 \geq 0 \\
            l^{-r} &\text{ if } s = r+1 \\
            0 &\text{ else}.
        \end{cases}
    \end{equation*}

    We denote by $M$ the operator given by
    \begin{equation*}
        M := \left(1 - \frac{l}{l^2 - 1} \right) \cdot I + \frac{1}{l} \cdot M_L + \frac{1}{l^3 - l} \cdot M_L^2
    \end{equation*}
    where $I$ is the identity operator over the countable state space $\mathbb{Z}_{\geq 0}$.
\end{definition}

We define the following probability distribution over $\mathbb{Z}_{\geq 0}$.
\begin{definition} \label{defn:initial-prob}
    Let $h$ be any square-free polynomial. (With abuse of notation, we may regard $h$ as a square-free product of places of $K$). We denote by $\delta_{h}: \mathbb{Z}_{\geq 0} \to [0,1]$ the probability distribution defined as
    \begin{equation*}
        \delta_{h}(J) := \frac{\# \{\omega \in \Omega_h \; | \; \text{dim}_{\mathbb{F}_l} \Sel(E[l],\omega) = J \}}{\# \Omega_h}.
    \end{equation*}
    We also denote by $\delta_1: \mathbb{Z}_{\geq 0} \to [0,1]$ the probability distribution defined as
    \begin{equation*}
        \delta_1(J) := \frac{\# \{\omega \in \Omega_1 \; | \; \text{dim}_{\mathbb{F}_l} \Sel(E[l],\omega) = J \}}{\# \Omega_1}.
    \end{equation*}
\end{definition}

\begin{remark} \label{rem:initial-prob}
    The distribution $\delta_1$ is equal to the probability distribution of dimensions of Selmer groups of the subset of cyclic prime twist family of $E$ twisted by irreducible polynomials in $\mathcal{P}_0$ of arbitrarily large degree. More concretely, we have
    \begin{equation}
        \delta_1(J) = \lim_{n \to \infty} \frac{\#\{f \in \mathcal{P}_0(n) : \dim_{\mathbb{F}_l} \mathrm{Sel}_{1-\sigma_{l,f}}(A_f/K) = J\}}{\# \mathcal{P}_0(n)}.
    \end{equation}
    The equality follows from using the fact that twisting by an irreducible polynomial $f$ in $\mathcal{P}_0$ only alters the images of the local Kummer maps at places $v \mid \Delta_E$, and that the uniform distribution over the set of order $l$ cyclic global characters $\chi_f \in \mathrm{Hom}(\mathrm{Gal}(\overline{K}/K),\mu_l)$ ramified only at an irreducible polynomial $f \in \mathcal{P}_0(n)$ restricts to a uniform distribution (up to some power-saving error terms determined by effective Chebotarev density theorem, i.e. Theorem \ref{thm:chebotarev}) of Cartesian product of unramified order $l$ local characters over places $v \mid \Delta_E$.
\end{remark}

Combining two applications of Chebotarev density theorem, we obtain the following proposition which relates the variations of local Selmer structures and Lagrangian mod-$l$ Markov operators.
\begin{proposition}[c.f. Theorem 9.5 of \cite{KMR14}] \label{prop:cheb3}
    Let $E$ be an elliptic curve satisfying Condition \ref{condition}. Suppose $q$ satisfies $$q > \exp(12 \log l + 2 \log D_{E,l} + (6 \log l) \cdot \# \Sigma_E),$$
    where $D_{E,l}$ is an explicitly computable constant stated in Corollary \ref{cor:Selmer_vary}. Then for any $J \geq 0$, any fixed polynomial $h$, and any $i \geq 1$ such that $\mathcal{P}_k(i) > (\# \Sigma_E + \#\{g \in \mathcal{P}_k : g \mid h\})^A$ for some constant $A \geq 2$, we have
    \begin{equation}
        \Biggl| \frac{\sum_{\chi \in \Omega_h} \sum_{\{\mathfrak{v} \in \mathcal{P}_k(i) : \mathfrak{v} \not\in \Sigma_E(h)\}} \# \{\chi' \in \Omega_{\chi,\mathfrak{v}} : \mathrm{rk}(\chi') = J \}}{l(l-1) \cdot \# \Omega_h \cdot \#\{\mathfrak{v} \in \mathcal{P}_k(i) : \mathfrak{v} \not\in \Sigma_E(h)\}} - (M_L^k \delta_h)(J) \Biggr| < 16 \cdot D_{E,l} \cdot l^{4 + 3 \# \Sigma_E(h)} \cdot q^{-i/2}.
    \end{equation}
\end{proposition}
\begin{proof}
    The idea of the proof is analogous to the proof of \cite[Theorem 9.5]{KMR14} and \cite[Proposition 5.13]{Pa22}. However, the statement of the proposition assumes that the degrees $i$ in $\mathcal{P}_k(i)$ (using irreducible polynomials of such degree with which one twists the elliptic curve) is greater than $\lfloor 4 (\log n + \log \log q)^2/\log q \rfloor$. For computing the explicit rate of convergence, we need effective error terms for \cite[Proposition 5.13]{Pa22} that are applicable for both $i \leq \mathfrak{n}$ and $i > \mathfrak{n}$. This is the reason why we provide the proof of the proposition above, instead of directly citing \cite[Proposition 5.13]{Pa22}.

    Given a fixed $\chi \in \Omega_h$, we have $\# \Omega_{\chi, \mathfrak{v}} = l(l-1)$. Denote by $\delta_\chi: \mathbb{Z}_{\geq 0} \to [0,1]$ the probability distribution
    \begin{equation}
        \delta_\chi(J) := \begin{cases}
            1 &\text{ if } J = \mathrm{rk}(\chi), \\
            0 &\text{ otherwise}.
        \end{cases}
    \end{equation}
    Then we use Proposition \ref{prop:Selmer_vary} and Proposition \ref{prop:cheb2} to obtain
    \begin{equation}
        \Biggl| \frac{\sum_{\{\mathfrak{v} \in \mathcal{P}_k(i) : \mathfrak{v} \not\in \Sigma_E(\sigma)\}} \# \{\chi' \in \Omega_{\chi,\mathfrak{v}} : \mathrm{rk}(\chi') = J \}}{l(l-1) \cdot \#\{\mathfrak{v} \in \mathcal{P}_k(i) : \mathfrak{v} \not\in \Sigma_E(\sigma)\}} - (M_L^k \delta_\chi)(J) \Biggr| < 16 \cdot D_{E,l} \cdot l^{4 + 3 \# \Sigma_E(\sigma)} \cdot q^{-i/2}.
    \end{equation}
    We note that the technical condition $\mathcal{P}_k(i) > (\# \Sigma_E + \#\{g \in \mathcal{P}_k : g \mid h\})^A$ is required to make sure that effective Chebotarev density theorem is still applicable over the set $\{\mathfrak{v} \in \mathcal{P}_k(i) : \mathfrak{v} \not\in \Sigma_E(\sigma) \cup \{g \in \mathcal{P}_k : g \mid h\}\}$. Otherwise, one may run into a situation where the square-free polynomial $h$ is divisible by all irreducible polynomials of degree $i$, for which the effective Chebotarev density theorem for $i = k$ is not applicable. (Of course, one can avoid this easily by either taking $q$ to be sufficiently large, or let $i$ to be sufficiently large. This will become a relevant observation in the upcoming section.)
    
    We observe that $\delta_h = \sum_{\chi \in \Omega_h}\frac{1}{\# \Omega_h} \delta_\chi$. Taking summation over $\chi \in \Omega_h$ and dividing both sides of the inequality by $\# \Omega_h$ gives the desired statement of the proposition.
\end{proof}

\section{Proof of the main result}

\subsection{Contributions from designated subsets of square-free polynomials}

The following proposition establishes the upper bound on the contributions of distribution of twist families of Selmer groups over the sets $F_n^{[2]}(\mathbb{F}_q)$ and $F_n^{[3]}(\mathbb{F}_q)$.
\begin{proposition} \label{prop:Fn23-contribution}
    Let $d \geq 1$ be a positive integer. Suppose we have
    $$n > \exp(\exp(\exp(\exp({e + 2d \log l})))).$$
    (It is a stronger condition than that appearing in Theorem \ref{thm:exceptional_set}). Set $R_2 := \lfloor \log n/ \log \log \log n \rfloor$. Then there exists explicitly computable constants $\rho(d,l) > 1$ depending only on $d$ and $l$, and an explicitly computable constant $G_{d,E,l}$ depending only on $d$, $E$, and $l$ such that $R_1 = \rho(d,l) \log n$ and 
    \begin{equation}
        \frac{\sum_{f \in F_n^{[2]}(\mathbb{F}_q) \cup F_n^{[3]}(\mathbb{F}_q)} \# (\mathrm{Sel}_{1-\sigma_{l,f}}(A_f/K))^d}{\# F_n(\mathbb{F}_q)} \leq G_{d,E,l} \cdot \frac{\sqrt{\log n}}{n^{0.9}}.
    \end{equation}
\end{proposition}
\begin{proof}
    By Corollary \ref{cor:Selmer_vary}, we obtain the following upper bound for any $f \in F_n^{[2]}(\mathbb{F}_q)$:
    \begin{equation}
        \dim_{\mathbb{F}_l} \mathrm{Sel}_{1-\sigma_{l,f}}(A_f/K) \leq D_{E,l} + 2 \omega(f) \leq D_{E,l} + 2 R_1.
    \end{equation}
    By Theorem \ref{thm:exceptional_set}, we have
    \begin{align}  \label{eqn:Fn2}
    \begin{split}
        \frac{\sum_{f \in F_n^{[2]}(\mathbb{F}_q)} \# (\mathrm{Sel}_{1-\sigma_{l,f}}(A_f/K))^d}{\# F_n(\mathbb{F}_q)} &< 4 \cdot 2^{D_{E,l} + 2 R_1} \cdot n^{-(\log \log n)^{1 - (\log \log \log n)^{-1/2}}} \\
        &< 4 \cdot 2^{D_{E,l}}\cdot n^{\rho \log 4 - (\log \log n)^{1 - (\log \log \log n)^{-1/2}}} \\
        &< 4 \cdot 2^{D_{E,l}} \cdot n^{\rho \log 4 - e^{1/2 \cdot e^{e + 2d \log l}}}.
    \end{split}
    \end{align}
    The upper bound is non-trivial as long as $R_1$ can be chosen to be at least $\rho \log n$ for some $\rho > 1$ such that $\rho \log 4 - e^{1/2 \cdot e^{e + 2d \log l}} < -1$. We will make a choice of $\rho$ by looking at the contributions of distribution of twist families of Selmer groups over the set $F_n^{[3]}(\mathbb{F}_q)$.

    We consider the following subdivision of the set $F_n^{[3]}(\mathbb{F}_q)$:
    \begin{equation} \label{eqn:Fn3-basic}
        \#F_n^{[3]}(\mathbb{F}_q) = \sum_{k > R_1} \#\{f \in F_n^{[3]}(\mathbb{F}_q): \omega(f) = k\} + \sum_{k \leq R_2} \#\{f \in F_n^{[3]}(\mathbb{F}_q): \omega(f) = k\}.
    \end{equation}
    By Corollary \ref{cor:Selmer_vary}, we have for each $k$,
    \begin{equation} \label{eqn:Fn3}
        \sum_{\substack{f \in F_n(\mathbb{F}_q) \\ \omega(f) = k}} \# (\mathrm{Sel}_{1-\sigma_{l,f}}(A_f/K))^d \leq l^{d(D_{E,l} + 2k)} \cdot \frac{q^n}{n} \frac{(\log n + 2 - \log 2)^{k-1}}{(k-1)!}.    
    \end{equation}
    We first obtain the upper bound for the first term of equation (\ref{eqn:Fn3-basic}). 
    Using the lower bound for Stirling's approximation, we obtain the following upper bound for $k > R_1$:
    \begin{align*}
        (\ref{eqn:Fn3}) &\leq l^{d(D_{E,l} + 2k)} \frac{q^n}{n} \left(1 + \frac{2 - \log 2}{\log n} \right)^{k-1} \cdot \left( \frac{\log n}{k} \right)^{k-1} \cdot e^k \cdot \frac{1}{\sqrt{2 \pi k}} \\
        &\leq l^{d(D_{E,l} + 2k)} \frac{q^n}{n} e^{(k-1)(2 - \log 2)/\log n} \cdot e^{(k-1)(\log \log n - \log k)} e^k \frac{1}{\sqrt{2 \pi k}} \\
        &\leq \frac{q^n}{n \cdot \log n \cdot \sqrt{2 \pi k}} \cdot \exp \left( -k (\log k - \log \log n - 3 - 2d \log l) + \log k + d D_{E,l} \log l \right).
    \end{align*}
    Hence as long as we choose
    \begin{equation*}
        R_1 \geq e^{4 + 3d \log l} \cdot \log(n), \hspace{15pt} (i.e. \rho(d,l) = e^{4 + 3d \log l}),
    \end{equation*}
    we obtain that 
    \begin{equation*}
    (\ref{eqn:Fn3}) \leq \frac{q^n}{n \log n \sqrt{2 \pi k}} \exp \left( -k d \log l + d D_{E,l} \log l \right) = \frac{q^n}{n \log n} \cdot \frac{l^{dD_{E,l}}}{l^{-dk} \sqrt{2 \pi k}}
    \end{equation*}
    for any $k > R_1$. Furthermore, such a choice of $\rho$ ensures that 
    \begin{equation*}
        (\ref{eqn:Fn2}) < 4 \cdot 2^{D_{E,l}} \cdot n^{-1877}.
    \end{equation*}
    We can then take summation over all $R_1 < k \leq n$ to get a uniform upper bound (recall that a polynomial of degree $n$ cannot have more than $n$ many irreducible factors)
    \begin{equation}
        \frac{\sum_{f \in F_n^{[3]}(\mathbb{F}_q), \omega(f) > R_1} \#(\mathrm{Sel}_{1-\sigma_{l,f}}(A_f/K))^d}{\# F_n(\mathbb{F}_q)} \leq \frac{1}{n \log n} \cdot \sum_{R_1 < k \leq n} \frac{l^{d D_{E,l}}}{l^{-dk}} \leq \frac{l^{d D_{E,l}}}{n \log n}.
    \end{equation}
    Next, we look at the second term of equation (\ref{eqn:Fn3-basic}). We obtain the following upper bound for $k \leq R_2$:
    \begin{align*}
        (\ref{eqn:Fn3}) &\leq l^{d (D_{E,l} + 2k)} \cdot \frac{q^n}{n} \left(1 + \frac{2 - \log 2}{\log n} \right)^{k-1} \cdot \left( \frac{\log n}{k} \right)^{k-1} \cdot e^k \cdot \frac{1}{\sqrt{2\pi k}} \\
        &\leq l^{d D_{E,l}} \cdot \frac{q^n}{n \sqrt{2\pi k}} \exp \left( (2-\log 2) \frac{k-1}{\log n} + 2kd \log l + (k-1)(\log \log n - \log k) \right)
    \end{align*}
    We use our choice $R_2 = \lfloor \log n / \log \log \log n \rfloor$ and the supposition $\log \log \log n > e^{e + 2d \log l}$ to obtain (i.e. we use $\log \log \log m_{n,q} > e$ to show that $(k-1)(\log \log n - \log k)$ achieves its maximum value at $k = R_2$ over the range $k \in [1,R_2]$)
    \begin{align*}
        (\ref{eqn:Fn3}) &\leq l^{d D_{E,l}} \cdot e^{5 - 2 \log 2} \cdot \frac{q^n}{n \sqrt{2 \pi k} \log \log \log n} \cdot \exp \left( k (\log \log \log \log n + 2d \log l)  \right). \\
        &\leq l^{d D_{E,l}} \cdot e^{5 - 2 \log 2} \cdot \frac{q^n}{n \sqrt{2 \pi k} \log \log \log n} \cdot e^{2k \log \log \log \log n}.
    \end{align*}
    We use the trivial bound $\frac{\log \log \log \log n}{\log \log \log n} < 0.1$ when $\log \log \log n > e^{e + 2 d \log l} \geq e^{e+1}$ and $k < \log n$ to obtain
    \begin{align*}
        (\ref{eqn:Fn3}) &\leq l^{d D_{E,l}} \cdot e^{5 - 2 \log 2} \cdot \frac{q^n}{n^{0.9} \sqrt{2 \pi k} \log \log \log n}.
    \end{align*}
    We can then take summation over all $1 \leq k < R_2$ to get a uniform upper bound
    \begin{align}
    \begin{split}
        \frac{\sum_{f \in F_n^{[3]}(\mathbb{F}_q), \omega(f) \leq R_2} \#(\mathrm{Sel}_{1-\sigma_{l,f}}(A_f/K))^d}{\# F_n(\mathbb{F}_q)} &\leq \sum_{1 \leq k \leq R_2} l^{d D_{E,l}} \cdot e^{5 - 2 \log 2} \cdot \frac{1}{n^{0.9} \sqrt{2 \pi k} \log \log \log n} \\
        &< \frac{1}{5} l^{d D_{E,l}} \cdot \frac{\sqrt{\log n}}{n^{0.9}}.
    \end{split}
    \end{align}
    Here we use the trivial lower bound $\log \log \log n \geq e^{e+1}$ to get
    \begin{align*}
        \sum_{1 \leq k \leq R_2} \frac{e^{5 - 2 \log 2}}{\sqrt{2 \pi k} \log \log \log n} < 2 \sqrt{\frac{\log n}{\log \log \log n}} \cdot \frac{e^{5-2\log 2}}{\sqrt{2 \pi} \log \log \log n} < \frac{1}{5} \sqrt{\log n}.
    \end{align*}
    This allows us to conclude that the contribution from the subset $F_n^{[3]}(\mathbb{F}_q)$ is of order
    \begin{equation}
        \frac{\sum_{f \in F_n^{[3]}(\mathbb{F}_q)} \#(\mathrm{Sel}_{1-\sigma_{l,f}}(A_f/K))^d}{\# F_n(\mathbb{F}_q)} \leq \frac{6}{5} \frac{l^{d D_{E,l}} \sqrt{\log n}}{n^{0.9}}.
    \end{equation}
    It hence follows that the contribution from the subset $F_n^{[2]}(\mathbb{F}_q)$ is of order smaller than that from the subset $F_n^{[3]}(\mathbb{F}_q)$. We can hence take $G_{d,E,l} = 4D_{E,l} + 8 + \frac{6}{5} l^{d D_{E,l}}$ to prove the proposition.
\end{proof}

Starting from this section and onwards, we will use
\begin{equation}
    R_1 := e^{4 + 3 d \log l} \log(n), \hspace{15pt} R_2 := \lfloor \frac{\log n}{\log \log \log n} \rfloor
\end{equation}
to define the sets $F_n^{main}(\mathbb{F}_q)$ and $F_n^{[1]}$. We turn our focus to obtaining contributions over the sets $F_n^{main}(\mathbb{F}_q)$ and $F_n^{[1]}(\mathbb{F}_q)$. To do so, we define the following subsets of $F_n^{main}(\mathbb{F}_q)$ and $F_n^{[1]}(\mathbb{F}_q)$, which will be used in the upcoming reinterpretation of the main result from \cite{Pa22}.
\begin{definition}
    We define two disjoint subsets of $F_n^{main}(\mathbb{F}_q)$:
    \begin{itemize}
        \item $F_n^{main,0}(\mathbb{F}_q) := \Big\{ f \in F_n^{main}(\mathbb{F}_q): f \text{ has an irred. factor in } \mathcal{P}_0 \Big\}$,
        \item $F_n^{main,1}(\mathbb{F}_q) := \Big\{ f \in F_n^{main}(\mathbb{F}_q): f \text{ has no irred. factors in } \mathcal{P}_0 \Big\}$.
    \end{itemize}
\end{definition}
\begin{definition}
We define the following subsets of polynomials of degree $n$ over $\mathbb{F}_q$ given a fixed choice of an elliptic curve $E$ satisfying Condition \ref{condition}.
\begin{itemize}
    \item We denote by $\mathcal{Q}(n)$ the set of squarefree polynomials of degree at most $n$ whose irreducible factors are all of degree at most $\mathfrak{n}$, and whose number of irreducible factors are at most $R_2$.
    \item Let $h$ be a square-free polynomial in $\mathcal{Q}(n)$, and $n_0 > \mathfrak{n}$ be some positive integer. We denote by $F_n^{main,0,\langle h, n_0 \rangle}(\mathbb{F}_q)$ a subset of $F_n^{main,0}(\mathbb{F}_q)$ consisting of polynomials $f$ such that $f_* = h$, and $n_0$ is the maximal degree of an irreducible factor $g_a$ of $f$ lying in $\mathcal{P}_0$. (We note that such an irreducible factor $g_a$ was called an auxiliary place of $f$, see \cite[Definition 5.8]{Pa22}.)
    \item Let $h$ be a square-free polynomial in $\mathcal{Q}(n)$, and $n_0 > \mathfrak{n}$ be some positive integer. We denote by $F_n^{main,1,\langle h, n_0 \rangle}(\mathbb{F}_q)$ a subset of $F_n^{main,1}(\mathbb{F}_q)$ consisting of polynomials $f$ such that $f_* = h$, and $n_0$ is the maximal degree of an irreducible factor $g_a$ of $f$ lying in $\mathcal{P}_1$.
    \item Let $h$ be a square-free polynomial in $\mathcal{Q}(n)$, and $n_0 > \mathfrak{n}$ be some positive integer. We denote by $F_n^{[1],\langle h, n_0 \rangle}(\mathbb{F}_q)$ a subset of $F_n^{[1]}(\mathbb{F}_q)$ consisting of polynomials $f$ such that $f_* = h$, and $n_0$ is the maximal degree of an irreducible factor $g_a$ of $f$ not dividing $h$. (By definition such $g_a$ lies in $\mathcal{P}_2$).
\end{itemize}
By construction we have
\begin{align*}
    F_n^{main,0}(\mathbb{F}_q) &= \bigsqcup_{h \in \mathcal{Q}(n)} \bigsqcup_{n_0 = \mathfrak{n}+1}^{n - \deg h} F_n^{main,0,\langle h,n_0 \rangle}(\mathbb{F}_q), \\
    F_n^{main,1}(\mathbb{F}_q) &= \bigsqcup_{h \in \mathcal{Q}(n)} \bigsqcup_{n_0 = \mathfrak{n}+1}^{n - \deg h} F_n^{main,1,\langle h,n_0 \rangle}(\mathbb{F}_q), \\
    F_n^{[1]}(\mathbb{F}_q) &= \bigsqcup_{h \in \mathcal{Q}(n)} \bigsqcup_{n_0 = \mathfrak{n}+1}^{n - \deg h} F_n^{[1],\langle h,n_0 \rangle}(\mathbb{F}_q).
\end{align*}
\end{definition}
The three series of propositions allow us to obtain contributions over the sets $F_n^{main,0}(\mathbb{F}_q)$, $F_n^{main,1}(\mathbb{F}_q)$, and $F_n^{[1]}(\mathbb{F}_q)$ using Lagrangian mod-$l$ operators $M$. We first start with contributions over $F_n^{main,0}(\mathbb{F}_q)$.
\begin{proposition} \label{prop:reinterpret-m0}
    Let $E$ be an elliptic curve which satisfies Condition \ref{condition}. Given a fixed $h \in \mathcal{Q}(n)$ and any $f \in F_n^{main,0,\langle h, n_0\rangle}(\mathbb{F}_q)$ for some $n_0 > \mathfrak{n}$, we denote by $\beta(l,f,h,d)$ the random variable 
\begin{equation*}
    \beta_{main,0}(l,f,h,d) := l^{d \cdot \left(M^{\omega(f^*)-1}(\delta_h) \right)}.
\end{equation*}
We denote by $\mathbb{E}[\beta_{main,0}(l,f,h,d) \; | \; F_n^{main,0,\langle h,n_0 \rangle}(\mathbb{F}_q)]$ the $d$-th moments of the random variable
\begin{align*}
    \mathbb{E}[\beta_{main,0}(l,f,h,d) \; | \; F_n^{main,0,\langle h,n_0 \rangle}(\mathbb{F}_q)] &:= \frac{\sum_{f \in F_n^{main,0,\langle h,n_0 \rangle}(\mathbb{F}_q)} \beta_{main,0}(l,f,h,d)}{\# F_n^{main,0,\langle h,n_0 \rangle}(\mathbb{F}_q)}.
\end{align*}
Suppose we have $\log \log \log n > e^{e + 2d \log l}$ (the same condition appearing in Proposition \ref{prop:Fn23-contribution}). Then we have
\begin{align}
\begin{split}
& \hspace{15pt} {\sum_{f \in F_n^{main,0}(\mathbb{F}_q)} \#(\mathrm{Sel}_{1-\sigma_{l,f}}(A_f/K))^d} \\
&= \sum_{h \in \mathcal{Q}(n)} \sum_{n_0 = \mathfrak{n}+1}^{n-\deg h} \# F_n^{main,0,\langle h,n_0\rangle}(\mathbb{F}_q) \cdot \left[ \mathbb{E}[\beta_{main,0}(l,f,h,d) \; | \; F_n^{main,0,\langle h,n_0 \rangle}(\mathbb{F}_q)] + O_{E,l,d} \left( n^{-\log n} \right) \right],
\end{split}
\end{align}
where the implied constant depends only on $E, l,$ and $d$, and is independent of $n$ and $q$.
\end{proposition}
\begin{proof}
The idea of the proof follows through \cite[Proposition 5.13]{Pa22}. However, similar to the proof of Proposition \ref{prop:cheb3}, we directly provide the proof of the theorem above, instead of directly citing it.

By construction, any polynomial in $F_n^{main}(\mathbb{F}_q) \cup F_n^{[1]}(\mathbb{F}_q)$ has an irreducible factor of degree greater than $\mathfrak{n}$. Given a square-free polynomial $h \in \mathcal{Q}(n)$, a fixed $n_0 > \mathfrak{n}$, and a fixed $R_2 < k \leq R_1$, we consider the subset
\begin{equation*}
    \{f \in F_n^{main,0,\langle h, n_0 \rangle}(\mathbb{F}_q) : \omega(f) = k\}.
\end{equation*}
There is a natural projection map
\begin{align*}
    \{f \in F_n^{main,0,\langle h, n_0 \rangle}(\mathbb{F}_q) : \omega(f) = k\} &\to \Omega_{f/g_a}.
\end{align*}
By Proposition \ref{prop:cheb1}, the pushforward of the uniform distribution over the above two subsets with respect to the natural projection map converges to the uniform distribution over the Cartesian product of the set (or a subset in case $l \leq 3$ and $K(\{\sqrt[l]{g} : g \mid f \text{ irreducible}, g \neq g_a\}) \cap K(E[l]) \neq K$) of cyclic order-$l$ characters $\Omega_{f/g_a}$. Both error terms (or rates of convergence) are bounded above by at most 
\begin{equation*}
    6(2 g_{E[l]} + 2l^3 - l - 2) \cdot n^{-2 \log n + e^{4 + 3d \log l} \cdot \log l}
\end{equation*}
because the degree of $g_a$ is greater than $\mathfrak{n}$. (In particular, we use Proposition \ref{prop:cheb1} by setting $i = n_0 > \mathfrak{n}  > \frac{\mathfrak{n}}{\log q}$). Notice that the pushforward of the uniform distribution does not imply uniform distribution over the set of cyclic order-$l$ characters in $\mathrm{Hom}(\mathrm{Gal}(\overline{K}_{g_a}/K_{g_a}), \mu_l)$. Nevertheless, because $g_a \in \mathcal{P}_0$, we may use Proposition \ref{prop:Selmer_vary} to obtain that altering local conditions at the place $g_a$ does not affect the dimension of local Selmer structures. Therefore, using the uniform distribution over the Cartesian product of the set (or subset) of cyclic order-$l$ characters $\Omega_{f/g_a}$, we can use Proposition \ref{prop:cheb3} by $k - \omega(h)$ many times to obtain error terms bounded above by at most
\begin{align*}
    &\leq 6(2 g_{E[l]} + 2l^3 - l - 2) \cdot n^{-2 \log n + e^{4 + 3d \log l} \cdot \log l} + R_1 \cdot 16 \cdot D_{E,l} \cdot l^{4 + 3\# \Sigma_E + R_2} \cdot n^{-2 \log n} \\
    &\leq \left(6(2 g_{E[l]} + 2l^3 - l - 2) + 16 \cdot D_{E,l} \cdot e^{4 + (4 + 3d + 3 \# \Sigma_E) \log l} \right) \cdot n^{-2 \log n + e^{4 + 3d \log l} \cdot \log l + \frac{\log \log n}{\log n} + \frac{\log l}{\log \log \log n}}.
\end{align*}
With the condition $\log \log \log > e^{e + 2d \log l}$, it follows that the above expression can be simplified into
\begin{align*}
    &< \left(6(2 g_{E[l]} + 2l^3 - l - 2) + 16 \cdot D_{E,l} \cdot e^{4 + (4 + 3d + 3 \# \Sigma_E) \log l} \right) \cdot n^{-\log n},
\end{align*}
which proves the desired results.
\end{proof}
Next, we understand contributions over the subset $F_n^{main,1}(\mathbb{F}_q)$.
\begin{proposition} \label{prop:reinterpret-m1}
    Let $E$ be an elliptic curve which satisfies Condition \ref{condition}. Denote by $M_1$ the Markov operator over $\mathbb{Z}_{\geq 0}$ defined as
    \begin{equation}
        M_1 := \frac{l^2-1}{l^2} M_L + \frac{1}{l^2} M_L^2.
    \end{equation}
    Given a fixed $h \in \mathcal{Q}(n)$ and any $f \in F_n^{main,1,\langle h, n_0\rangle}(\mathbb{F}_q)$ for some $n_0 > \mathfrak{n}$, we denote by $\beta_{main,1}(l,f,h,d)$ the random variable 
\begin{equation*}
    \beta_{main,1}(l,f,h,d) := l^{d \cdot \left(M_L \cdot M_1^{\omega(f^*)-1}(\delta_h) \right)}.
\end{equation*}
We denote by $\mathbb{E}[\beta_{main,1}(l,f,h,d) \; | \; F_n^{main,1,\langle h,n_0 \rangle}(\mathbb{F}_q)]$ the $d$-th moments of the random variable
\begin{align*}
    \mathbb{E}[\beta_{main,1}(l,f,h,d) \; | \; F_n^{main,1,\langle h,n_0 \rangle}(\mathbb{F}_q)] &:= \frac{\sum_{f \in F_n^{main,1,\langle h,n_0 \rangle}(\mathbb{F}_q)} \beta_{main,0}(l,f,h,d)}{\# F_n^{main,1,\langle h,n_0 \rangle}(\mathbb{F}_q)}.
\end{align*}
Suppose we have $\log \log \log n > e^{e + 2d \log l}$ (the same condition appearing in Proposition \ref{prop:Fn23-contribution}). Then we have
\begin{align}
\begin{split}
& \hspace{15pt} {\sum_{f \in F_n^{main,1}(\mathbb{F}_q)} \#(\mathrm{Sel}_{1-\sigma_{l,f}}(A_f/K))^d} \\
&= \sum_{h \in \mathcal{Q}(n)} \sum_{n_0 = \mathfrak{n}+1}^{n-\deg h} \# F_n^{main,1,\langle h,n_0\rangle}(\mathbb{F}_q) \cdot \left[ \mathbb{E}[\beta_{main,1}(l,f,h,d) \; | \; F_n^{main,1,\langle h,n_0 \rangle}(\mathbb{F}_q)] + O_{E,l,d} \left( n^{-\log n} \right) \right],
\end{split}
\end{align}
where the implied constant depends only on $E, l,$ and $d$, and is independent of $n$ and $q$.
\end{proposition}
\begin{proof}
    The proof is a line-by-line adaptation of the proof of Proposition \ref{prop:reinterpret-m0}. The only difference is that our auxiliary place $g_a$ now lies in $\mathcal{P}_1$ instead of $\mathcal{P}_0$. But Proposition \ref{prop:Selmer_vary} implies that $t_\chi(g_a)$ (not a specific choice of the cyclic order-$l$ character unlike the case when $g_a \in \mathcal{P}_2$) determines the variation of dimensions of local Selmer structures. Therefore, using the uniform distribution over $\Omega_{f/g_a}$, we can use Proposition \ref{prop:cheb3} by $k - \omega(h)$ many times, and then use Proposition \ref{prop:cheb3} once over the set of primes $\mathcal{P}_1(n_0)$ to obtain that the dimension of Selmer groups $\mathrm{Sel}_{1-\sigma_{l,f}}(A_f/K)$ as $f$ varies over $F_n^{main,1}(\mathbb{F}_q)$ is governed by the Markov operator
    \begin{equation}
        M_L \cdot \left( \frac{l^2-1}{l^2} M_L + \frac{1}{l^2} M_L^2 \right)^{\omega(f^*)-1} (\delta_h).
    \end{equation}
    Note that the weights $\frac{l^2 - 1}{l^2}$ and $\frac{1}{l^2}$ are obtained from using effective Chebotarev density theorem to compute the conditional probability that a place in $\mathcal{P}_1 \cup \mathcal{P}_2$ lies in $\mathcal{P}_1$ (or respectively $\mathcal{P}_2$). Recall that the subset $F_n^{main,1}(\mathbb{F}_q)$ consists of degree $n$ polynomials with no irreducible factors in $\mathcal{P}_0$. This is the reason why $\beta_{main,1}(l,f,h,d)$ is defined with respect to $M_1$, instead of $M$ as in the previous proposition. The procedure of computing the error term is analogous to the proof of Proposition \ref{prop:reinterpret-m0}.
\end{proof}
Lastly, we understand contributions from the subset $F_n^{[1]}(\mathbb{F}_q)$. As hinted in the proof of the previous proposition, it is not obvious how one can obtain an explicit relation between moments of local Selmer structures over the set $F_n^{[1]}(\mathbb{F}_q)$ and Lagrangian mod-$l$ Markov operators. Nevertheless, it is still possible to obtain upper bound on the moments, which we will show is enough for obtaining the desired rate of convergence.
\begin{proposition} \label{prop:reinterpret-[1]}
    Let $E$ be an elliptic curve which satisfies Condition \ref{condition}. Given a fixed $h \in \mathcal{Q}(n)$ and any $f \in F_n^{[1],\langle h, n_0\rangle}(\mathbb{F}_q)$ for some $n_0 > \mathfrak{n}$, we denote by $\beta_{[1]}(l,f,h,d)$ the random variable 
\begin{equation*}
    \beta_{[1]}(l,f,h,d) := l^{d \cdot \left(M_L^{2(\omega(f^*)-1)}(\delta_h) \right)}.
\end{equation*}
We denote by $\mathbb{E}[\beta_{[1]}(l,f,h,d) \; | \; F_n^{[1],\langle h,n_0 \rangle}(\mathbb{F}_q)]$ the $d$-th moments of the random variable
\begin{align*}
    \mathbb{E}[\beta_{[1]}(l,f,h,d) \; | \; F_n^{[1],\langle h,n_0 \rangle}(\mathbb{F}_q)] &:= \frac{\sum_{f \in F_n^{[1],\langle h,n_0 \rangle}(\mathbb{F}_q)} \beta_{[1]}(l,f,h,d)}{\# F_n^{[1],\langle h,n_0 \rangle}(\mathbb{F}_q)}.
\end{align*}
Suppose we have $\log \log \log n > e^{e + 2d \log l}$ (the same condition appearing in Proposition \ref{prop:Fn23-contribution}). Then we have
\begin{align}
\begin{split}
& \hspace{15pt} {\sum_{f \in F_n^{[1]}(\mathbb{F}_q)} \#(\mathrm{Sel}_{1-\sigma_{l,f}}(A_f/K))^d} \\
&\leq l^{2d} \cdot \sum_{h \in \mathcal{Q}(n)} \sum_{n_0 = \mathfrak{n}+1}^{n-\deg h} \# F_n^{[1],\langle h,n_0\rangle}(\mathbb{F}_q) \cdot \left[ \mathbb{E}[\beta_{[1]}(l,f,h,d) \; | \; F_n^{[1],\langle h,n_0 \rangle}(\mathbb{F}_q)] + O_{E,l,d} \left( n^{-\log n} \right) \right],
\end{split}
\end{align}
where the implied constant depends only on $E, l,$ and $d$, and is independent of $n$ and $q$.
\end{proposition}
\begin{proof}
    Again, the proof is a line-by-line adaptation of the proof of Proposition \ref{prop:reinterpret-m0}. The only difference is that our auxiliary place $g_a$ lies in $\mathcal{P}_2$ instead of $\mathcal{P}_0$. Proposition \ref{prop:Selmer_vary} implies that the dimension of local Selmer structures can increase by at most $2$ after changing the local conditions at the place $g_a$. Hence, using the uniform distribution over $\Omega_{f/g_a}$, we can use Proposition \ref{prop:cheb3} by $k - \omega(h)$ many times, and then multiply the resulting distribution by $l^{2d}$ to obtain the desired upper bound. The procedure of computing the error term is analogous to the proof of Proposition \ref{prop:reinterpret-m0}. The governing operator over the subset $F_n^{[1]}(\mathbb{F}_q)$ is $M_L^2$, because all irreducible factors of $f \in F_n^{[1]}(\mathbb{F}_q)$ lie in $\mathcal{P}_2$. This is the reason why $\beta_{[1]}(l,f,h,d)$ is defined in terms of $M_L^2$, instead of $M$ or $M_1$ as in previous propositions.
\end{proof}

\subsection{Secondary term for first moments}
\label{section:lowerbound}

The goal of this section is to explicitly compute the secondary term for the first moment of $1-\sigma_{l,f}$ Selmer groups of $A_f$ (i.e. $d = 1$). The key tool we will use to obtain the desired lower bound is to obtain explicit rate of convergence of the Markov operator governing the variations of $1-\sigma_{l,f}$ Selmer groups from Definition \ref{defn:Markov}, which is proved in \cite[Corollary 6.7]{Pa22}.

\begin{theorem}[Corollary 6.7 of \cite{Pa22}] \label{thm:convergence}
    Let $M$ be the Markov operator constructed in Definition \ref{defn:Markov}. Given a probability distribution $\delta: \mathbb{Z}_{\geq 0} \to [0,1]$, denote by $\beta(\delta,l,d)$ the random variable
    \begin{equation*}
        \beta(\delta,l,d) := l^{d \delta}.
    \end{equation*}
    Set $d = 1$. Then there exists an explicitly computable constant $\kappa_d$ such that the following recursive relation holds:
    \begin{equation}
        \mathbb{E}[\beta(M\delta,l,1)] = \mathbb{E}[\beta(\delta,l,1)] \cdot \left( 1 - \frac{l^2 - l + 1}{l^3} \right) + \left( 1 + \frac{1}{l^3} \right).
    \end{equation}
\end{theorem}
From now on, we set
\begin{equation}
    \gamma_l := 1 - \frac{l^2 - l + 1}{l^3}, \hspace{15pt} \kappa_l := 1 + \frac{1}{l^3}.
\end{equation}
We note that iterating Theorem \ref{thm:convergence} by $k$ times gives
\begin{align} \label{eqn:important}
    \begin{split}
        \mathbb{E}[\beta(M^k \delta, l, 1)] &= \mathbb{E}[\beta(\delta,l,1)] \cdot \gamma_l^k + \kappa_l \cdot \sum_{j=0}^{k-1} \gamma_l^j.
    \end{split}
\end{align}
Denote by $PR_l: \mathbb{Z}_{\geq 0} \to [0,1]$ the Poonen-Rains distribution which is defined as
    \begin{equation}
        PR_l(n) := \prod_{j \geq 0}^{\infty} \frac{1}{1 + l^{-j}} \cdot \prod_{j=1}^n \frac{l}{l^j - 1}.
    \end{equation}
Note that $\mathbb{E}[PR_l] = l+1$. As shown in \cite{KMR14} and \cite{Pa22}, the Poonen-Rains distribution is the stationary distribution of the Markov operator $M$. By using (\ref{eqn:important}) and allow $k$ to grow arbitrarily large, we can obtain an explicit formula for the constant $\kappa_l$ as follows:
    \begin{equation}
        \kappa_l = \mathbb{E}[PR_l] \cdot (1- \gamma_l).
    \end{equation}
Because $l$ is a positive number, the constant $\kappa_l$ is positive regardless of the choice of $l$. By substituting the explicit formula for $\kappa_l$ into (\ref{eqn:important}), we obtain
\begin{align} \label{eqn:important2}
\begin{split}
    \mathbb{E}[\beta(M^k \delta, l, 1)] - \mathbb{E}[PR_l] &= \left(\mathbb{E}[\beta(\delta,l,1)] - (l+1)\right) \cdot \gamma_l^k.
\end{split}
\end{align}
This implies that $\mathbb{E}[\beta(\delta,l,1)] < \mathbb{E}[PR_l]$ if and only if $\mathbb{E}[\beta(M^k \delta, l, 1)] < \mathbb{E}[PR_l]$ for every $k \geq 0$. 
\begin{remark} \label{remark:generalization_d}
    For $d \geq 2$, the recursive relation involves all moments up to $d$-th moments. To elaborate, one has
    \begin{equation} \label{eqn:important3}
        \mathbb{E}[\beta(M_L \delta, l, d)] = \mathbb{E}[\beta(\delta,l,d)] \cdot \frac{1}{l^d} + \left( l^d - \frac{1}{l^d} \right) \mathbb{E}[\beta(\delta,l,d-1)].
    \end{equation}
    where we simply regard $\mathbb{E}[\beta(\delta, l, 0)] = 1$. From this one obtains 
    \begin{align}
    \begin{split}
        \mathbb{E}[\beta(M \delta, l, d)] &= \gamma_{l,d} \cdot \mathbb{E}[\beta(\delta,l,d)] + \frac{(l^d+1)^2}{l^{2d+1}} \cdot \left(\sum_{s=0}^{d-1} l^s \right) \cdot \mathbb{E}[\beta(\delta,l,d-1)] \\
        &\hspace{15pt} + \frac{(l^{2d}-1)(l^{2d-2}-1)}{l^{2d}(l^2-1)} \cdot \mathbb{E}[\beta(\delta,l,d-2)],
    \end{split}
    \end{align}
    where $$\gamma_{l,d} = 1 - \frac{l}{l^2-1} + \frac{1}{l^{d+1}} + \frac{1}{l^{2d}(l^3-l)}.$$
    We note that $$1 - \frac{l}{l^2-1} < \gamma_{l,d+1} < \gamma_{l,d} \leq 1 - \frac{l^2 - l + 1}{l^3}$$ for every $d \geq 1$.
    Repeating the operator $M$ by $k$ many times such that $k > d$, we obtain that there exist explicit computable positive constants $\{\kappa_{s}(k,d,l)\}_{s=0}^d$ depending on $k, d, $ and $l$ but independent of $\gamma_{l,d}$ such that
    \begin{equation}
        \mathbb{E}[\beta(M^k \delta, l, d)] = \kappa_0(k,l,d) \cdot \sum_{j=0}^{k-1} \gamma_{l,1}^j + \sum_{s=1}^d \kappa_s(k,l,d) \cdot \gamma_{l,s}^k \cdot \mathbb{E}[\beta(\delta,l,s)],
    \end{equation}
    We note that terms of form $c \cdot \prod_{t \geq s} \gamma_{l,t}^{k_t} \cdot \mathbb{E}[\beta(\delta,l,s)]$ for some constant $c > 0$ and $k_s + k_{s+1} + \cdots + k_d \leq k$ do occur as terms in the equation above. But we can subsume such terms inside the constant term $\kappa_s(k,l,d)$ by rewriting such terms as $c' \cdot \gamma_{l,s}^k \cdot \prod_{t > s} (\gamma_{l,t} / \gamma_{l,s})^{k_t} \cdot \mathbb{E}[\beta(\delta,l,s)]$. (Note that $\gamma_{l,t} / \gamma_{l,s} < 1$ for all $t > s$). Similar to the case when $d = 1$, the term $\kappa_0(k,l,d)$ converges to $\mathbb{E}[(PR_l)^d] \cdot (1-\gamma_{l,1})$ as $k$ grows arbitrarily large. Unlike the case where $d = 1$, however, the values of $\kappa_0(k,l,d)$ changes as $k$ varies. Such a change in constants prevents one from obtaining an explicit equation for $\mathbb{E}[\beta(M^k \delta,l,d)] - \mathbb{E}[(PR_l)^d]$ unlike the case for $d = 1$ as in equation (\ref{eqn:important2}). Nevertheless, one can deduce that 
    \begin{equation}
        \kappa_{0}(k,l,d) \to \mathbb{E}[(PR_l)^d] \cdot (1 - \gamma_{l,1}) \hspace{20pt} \text{ as } k \to \infty,
    \end{equation}
    from which one may deduce that for sufficiently large $k$,
    \begin{align} \label{eqn:generalization_d}
    \begin{split}
        \mathbb{E}[\beta(M^k \delta, l, d)] - \mathbb{E}[(PR_l)^d] &\to \left(\kappa_1(k,l,d) \cdot \mathbb{E}[\beta(\delta,l,1)] - \prod_{m=1}^d (l^m+1) \right) \cdot \gamma_{l,1}^k \\
        &\hspace{15pt} + \sum_{s=2}^d \kappa_s(k,l,d) \cdot \mathbb{E}[\beta(\delta,l,s)] \cdot \gamma_{l,s}^k, \hspace{20pt} \text{ as } k \to \infty.
    \end{split}
    \end{align}
\end{remark}

We return back to the case where $d = 1$. One can obtain explicit conditions when $\mathbb{E}[\beta(\delta,l,1)] \neq \mathbb{E}[PR_l]$, as well as the lower bound of rate of convergence.
\begin{corollary} \label{cor:convergence}
    Suppose $\delta: \mathbb{Z}_{\geq 0} \to [0,1]$ is a probability distribution such that satisfies the following two conditions.
    \begin{itemize}
    \item There are only finitely many $j \in \mathbb{Z}_{\geq 0}$ such that $\delta(j) \neq 0$.
    \item There exists a fixed positive number $L$ that is not divisible by $l-1$ such that for all $j$'s such that $\delta(j) \neq 0$, we have $\delta(j) = \frac{a_j}{L}$ for some positive integer $a_j$.
    \end{itemize}
    Then the following inequality holds for every $k \geq 1$:
    \begin{equation}
        \left| \mathbb{E}[\beta(M^k \delta, l, 1)] - \mathbb{E}[PR_l] \right| \geq \frac{1}{L} \cdot \gamma_l^k.
    \end{equation}
    Furthermore, the differences $\mathbb{E}[\beta(M^k \delta, l, 1)] - \mathbb{E}[PR_l]$ are all positive or all negative for every $k \geq 0$.
\end{corollary}
\begin{proof}
    Without loss of generality, assume that $\delta(j) \neq 0$ if and only if $j \in \{0, 1, \cdots, m\}$. The condition that $\delta$ is a probability distribution implies that
    \begin{equation}
        L \cdot \mathbb{E}[\beta(\delta,l,d)] = \sum_{j=0}^m l^{dj} a_j = L + \sum_{j=1}^m (l^{dj}-1) a_j = L + (l^d - 1) \cdot \left(\sum_{j=1}^m \frac{l^{dj}-1}{l^d - 1} a_j \right).
    \end{equation}
    We now choose $d = 1$. Suppose that $\mathbb{E}[\beta(\delta,l,1)] = \mathbb{E}[PR_l] = l+1$. Then we can subtract both sides of the equation above by $L$ to obtain
    \begin{equation*}
        L \cdot l = (l - 1) \cdot \left(\sum_{j=1}^m \frac{l^{j}-1}{l - 1} a_j \right).
    \end{equation*}
    But this is a contradiction because the left hand side is not divisible by $l-1$, but the right hand side is divisible by $l-1$. Therefore, $\mathbb{E}[\beta(\delta,l,1)] \neq \mathbb{E}[PR_l]$. The second condition on $\delta$ implies that $|\mathbb{E}[\beta(\delta,l,1)] - \mathbb{E}[PR_l]| \geq \frac{1}{L}$. We can hence use (\ref{eqn:important2}) to obtain the desired result as well as the uniform signature of the differences $\mathbb{E}[\beta(M^k \delta, l, 1)] - \mathbb{E}[PR_l]$ for every $k \geq 0$.
\end{proof}
\begin{remark}
We note that the statement of Corollary \ref{cor:convergence} can be generalized to any $d$-th moments of probability distribution as well, as long as $l^d - 1$ does not divide $\prod_{m=1}^d (l^m+1) - 1$. In such cases, one would need to change the second condition so that there exists a fixed positive number $L$ that is not divisible by $l^d - 1$ such that for all $j$'s such that $\delta(j) \neq 0$, we have $\delta(j) = \frac{a_j}{L}$ for some positive integer $a_j$.
\end{remark}

We obtain explicit rate of convergence results on collections of probability distributions $\delta_h: \mathbb{Z}_{\geq 0} \to [0,1]$ by using the probability distribution $\delta_1$, both probability distributions of which were defined in Definition \ref{defn:initial-prob}.

\begin{lemma} \label{lemma:init-prob}
    Let $h$ be a square-free polynomial satisfying the following three conditions:
    \begin{itemize}
        \item All irreducible factors of $h$ are at most $\mathfrak{n}$,
        \item $\omega(h) < R_2$,
        \item For all $i \geq 1$, we have $\mathcal{P}_k(i) > A' \cdot (\# \Sigma_E + \#\{g \in \mathcal{P}_k : g | h\})^A$ for some fixed constant $A \geq 2$ and $A' > 0$.
    \end{itemize}
    We define $\omega'(h)$ to be the quantity
    \begin{equation}
        \omega'(h) := \#\{\mathfrak{v} \in \mathcal{P}_1 : \mathfrak{v} \mid h\} + 2 \cdot \#\{\mathfrak{v} \in \mathcal{P}_2 : \mathfrak{v} \mid h\}.
    \end{equation}
    Suppose that $E/K$ is an elliptic curve satisfying Condition \ref{condition} and $\mathbb{E}[\beta(\delta_1, l, 1)] \neq \mathbb{E}[PR_l]$.
    Suppose $q$ satisfies $q > e^{\left(24 \log l + 4 \cdot \log D_{E,l} + 12 \cdot (\log l) \cdot \#\Sigma_E \right)}$.
    Then the distribution $\delta_h: \mathbb{Z}_{\geq 0} \to [0,1]$ satisfies
    \begin{align}
    \begin{split}
        \mathbb{E}[\beta(\delta_h,l,1)] - \mathbb{E}[PR_l] &= \left(\mathbb{E}[\beta(\delta_1, l, 1)] - \mathbb{E}[PR_l] \right) \cdot \left( \frac{1}{l} \right)^{\omega'(h)} + O_{E,l} \left( \log n \cdot n^{\frac{3 \# \Sigma_E \log l}{\log \log \log n}} \cdot q^{-1/2} \right),
    \end{split}
    \end{align}
    where the implied constant of the error term is independent of $n$ and $q$. Furthermore, for such values of $n$ and $q$, the differences $\mathbb{E}[\beta(\delta_h,l,1)] - \mathbb{E}[PR_l]$ and $\mathbb{E}[\beta(\delta_1,l,1)] - \mathbb{E}[PR_l]$ are either both positive or both negative.
\end{lemma}
\begin{proof}
    The main term follows from equation (\ref{eqn:important3}) with $d = 1$. We note that we cannot directly apply equation (\ref{eqn:important2}) because we chose $h$ as a given element in $\mathcal{Q}(n)$. The error term originates from applying Proposition \ref{prop:cheb3} to each iteration of the Markov operator $M$, from which the condition on the lower bound of $q$ is obtained (and it is in fact twice the exponent appearing in the condition on the lower bound of $q$). The implied constants of the error term is at most $16 D_{E,l} \cdot l^{4}$, which is independent of $n$ and $q$.
\end{proof}
\begin{remark} \label{rem:l5}
    In the case where $l \geq 7$, we have $l \neq 2^m + 1$ for any non-negative integer $m$. The probability distribution $\delta_h$ given a square-free polynomial $h$ satisfying conditions in Lemma \ref{lemma:init-prob} also satisfies both conditions of Corollary \ref{cor:convergence}. The cardinality of $\Omega_h$ is given by $(l^2-1)^{\# \Sigma_E} \cdot (l^2-l)^{\omega(h)}$. Nevertheless, we recall that a cyclic order-$l$ extension of a field gives rise to $l-1$ many cyclic order-$l$ characters. Hence, among such $l-1$ many characters, the dimensions of $\Sel(E[l],\chi)$ are identical. This implies that the non-zero values of probability distribution $\delta_h$ all have denominators of form $(l+1)^{\# \Omega_E} \cdot l^{\omega(h)}$ (without applying cancellation between numerators and denominators).
\end{remark}
\begin{remark} \label{rem:init-prob}
    Using the definition and an alternate simple expression for $\delta_1$ from Definition \ref{defn:initial-prob} and Remark \ref{rem:initial-prob}, we can compute $\mathbb{E}[\beta(\delta_1,l,1)]$ using the following asymptotic formula:
    \begin{equation}
        \mathbb{E}[\beta(\delta_1,l,1)] := \sum_{J=0}^\infty l^J \cdot \mathbb{P}[\delta_1 = J] = \lim_{n \to \infty} \sum_{J = 0}^\infty l^J \cdot \frac{\#\{f \in \mathcal{P}_0(n) : \dim_{\mathbb{F}_l} \mathrm{Sel}_{1-\sigma_{l,f}}(A_f/K) = J\}}{\# \mathcal{P}_0(n)}.
    \end{equation}
\end{remark}

We now obtain explicit rate of convergences for the contributions to distribution of Selmer groups over the three types of subsets $F_n^{main,0}(\mathbb{F}_q)$, $F_n^{main,1}(\mathbb{F}_q)$, and $F_n^{[1]}(\mathbb{F}_q)$. Given that all the analytic results we have used up to this point are uniform results which hold for sufficiently large $n$, we can now pick $q$ to be sufficiently large to simplify such explicit rate of convergences.
\begin{lemma} \label{lemma:convergence}
    Assume all conditions from Propositions \ref{prop:reinterpret-m0}, \ref{prop:reinterpret-m1}, \ref{prop:reinterpret-[1]}. Suppose $q$ satisfies
    \begin{equation}
        q > \max \left\{(R_2 + \deg \Delta_E)^2, \exp({24 \log l + 4 \log D_{E,l} + (12 \log l) \cdot \# \Sigma_E}) \right\}.
    \end{equation}
    We define
    \begin{equation}
        \gamma_{l,1} := \frac{l^3 - l + 1}{l^4}.
    \end{equation}
    Then the following equations hold for any $R_2 < k \leq R_1$ and any square-free polynomial $h \in \mathcal{Q}(n)$, where we recall $\omega'(h)$ as defined in Lemma \ref{lemma:init-prob}:
    \begin{align*}
        \begin{split}
            & \hspace{15pt} \frac{\sum_{\substack{f \in F_n^{main,0,\langle h,n_0 \rangle}(\mathbb{F}_q) \\ \omega(f) = k}} \left[\#\mathrm{Sel}_{1-\sigma_{l,f}}(A_f/K) - (l+1) \right]}{\# \{f \in F_n^{main,0,\langle h, n_0\rangle}(\mathbb{F}_q) : \omega(f) = k \}} \\
            &= \left( \mathbb{E}[\beta(\delta_1,l,1)] - (l+1) \right) \cdot \gamma_l^{k-\omega(h)-1} \cdot \frac{1}{l^{\omega'(h)}} + O_{E,l} \left( \log n \cdot n^{\frac{3 \# \Sigma_E \log l}{\log \log \log n}} \cdot \gamma_l^{k-\omega(h)-1} \cdot q^{-1/2} \right).
        \end{split}
    \end{align*}
    \begin{align*}
        \begin{split}
            & \hspace{15pt} \frac{\sum_{\substack{f \in F_n^{main,1,\langle h,n_0 \rangle}(\mathbb{F}_q) \\ \omega(f) = k}} \left[\#\mathrm{Sel}_{1-\sigma_{l,f}}(A_f/K) - (l+1) \right]}{\# \{f \in F_n^{main,1,\langle h, n_0\rangle}(\mathbb{F}_q) : \omega(f) = k \}} \\
            &= \left( \mathbb{E}[\beta(\delta_1,l,d)] - (l+1) \right) \cdot \gamma_{l,1}^{k-\omega(h)-1} \cdot \frac{1}{l^{\omega'(h)+1}} + O_{E,l} \left( \log n \cdot n^{\frac{3 \# \Sigma_E \log l}{\log \log \log n}} \cdot \gamma_{l,1}^{k-\omega(h)-1} \cdot q^{-1/2} \right).
        \end{split}
    \end{align*}
    \begin{align*}
        \begin{split}
            & \hspace{15pt} \frac{\sum_{\substack{f \in F_n^{[1],\langle h,n_0 \rangle}(\mathbb{F}_q) \\ \omega(f) = k}} \left[\#\mathrm{Sel}_{1-\sigma_{l,f}}(A_f/K)\right]}{\# \{f \in F_n^{[1],\langle h, n_0\rangle}(\mathbb{F}_q) : \omega(f) = k \}} \\
            &\leq l^2 \cdot \Biggl[(l+1) + \left( \mathbb{E}[\beta(\delta_1,l,1)] - (l+1) \right) \cdot \frac{1}{l^{2k-2\omega(h)+\omega'(h)-2}} + O_{E,l} \left( \log n \cdot n^{\frac{3 \# \Sigma_E \log l}{\log \log \log n}} \cdot \gamma_l^{k-\omega(h)-1} \cdot q^{-1/2} \right) \Biggr].
        \end{split}
    \end{align*}
\end{lemma}
\begin{proof}
    All these three expressions follow from substituting Lemma \ref{lemma:init-prob} into equation (\ref{eqn:important2}), and applying the proof of Proposition \ref{prop:reinterpret-m0} to the set $\{f \in F_n^{main,0,\langle h, n_0 \rangle}(\mathbb{F}_q) : \omega(f) = k\}$, proof of Proposition \ref{prop:reinterpret-m1} to the set $\{f \in F_n^{main,1,\langle h, n_0 \rangle}(\mathbb{F}_q) : \omega(f) = k\}$, and the proof of Proposition \ref{prop:reinterpret-[1]} to the set $\{f \in F_n^{[1],\langle h, n_0 \rangle}(\mathbb{F}_q) : \omega(f) = k\}$. The lower bound on $q$ ensures that all conditions appearing in Lemma \ref{lemma:init-prob} are satisfied.
    To apply Proposition \ref{prop:reinterpret-m1}, we use
    \begin{equation} \label{eq:lemma:convergence1}
        \mathbb{E}[\beta(M_L\delta,l,1)] = \mathbb{E}[\beta(\delta,l,1)] \cdot \frac{1}{l} + \left( l - \frac{1}{l} \right),
    \end{equation}
    and the relation
    \begin{equation} \label{eq:lemma:convergence2}
        \mathbb{E} \left[\beta \left(\left(\frac{l^2-1}{l^2}M_L + \frac{1}{l^2} M_L^2 \right)\delta,l,1 \right) \right] = \mathbb{E}[\beta(\delta,l,1)] \cdot \frac{l^3-l+1}{l^4} + \left( 1 + \frac{1}{l^3} \right) \cdot \left(l - \frac{1}{l} \right),
    \end{equation}
    which is where the constant $\gamma_{l,1}$ comes into play. Notice that using equation (\ref{eq:lemma:convergence2}) by $k - \omega(h) - 1$ many times and using equation (\ref{eq:lemma:convergence1}) for the last iteration gives the second statement of the lemma.
    
    When applying Proposition \ref{prop:reinterpret-[1]}, we use the relation
    \begin{equation}
        \mathbb{E}[\beta(M_L^2\delta,l,1)] = \mathbb{E}[\beta(\delta,l,1)] \cdot \frac{1}{l^2} + \frac{(l-1)(l+1)^2}{l^2}.
    \end{equation}
    The error term $O_{E,l}(n^{-\log n})$ gets subsumed into the error term computed in Lemma \ref{lemma:init-prob}.
\end{proof}

What remains to be computed is to take summations of all error terms in Lemma \ref{lemma:convergence} over the sets $F_n^{main,0}(\mathbb{F}_q)$, $F_n^{main,1}(\mathbb{F}_q)$, and $F_n^{[1]}(\mathbb{F}_q)$.
We first compute the contribution of rate of convergence over the set $F_n^{main,0}(\mathbb{F}_q)$.
\begin{proposition} \label{prop:convergence-m0}
    Assume conditions from Lemma \ref{lemma:convergence}. Suppose that $n$ and $q$ satisfy the following conditions for some $\epsilon > 0$:
    \begin{align}
        \begin{split}
            n &> \exp(\exp(\exp(\exp({e + 2 \log l})))), \\
            q &> \max \left\{\frac{n^{\frac{8 \# \Sigma_E \log l - 2}{\log \log \log n} + 2 \epsilon}}{(\log n)^2}, (R_2 + \deg \Delta_E)^2, \exp({24 \log l + 4 \log D_{E,l} + (12 \log l) \cdot \# \Sigma_E}) \right\}.
        \end{split}
    \end{align}
    Then the following equation holds with an explicitly computable constant $\beta_{E,l}$ such that for any $\epsilon > 0$,
    \begin{equation}
        \frac{\sum_{f \in F_n^{main,0}(\mathbb{F}_q)} (\# \mathrm{Sel}_{1-\sigma_{l,f}}(A_f/K) - (l+1))}{\# F_n(\mathbb{F}_q)} = \beta_{E,l} \cdot n^{-\frac{l^2 - l + 1}{l^3}} + O_\epsilon \left((\log n)^{-1/2 + \epsilon} \cdot n^{-\frac{l^2-l+1}{l^3}} \right).
    \end{equation}
    The implied constant of the error term is independent of $n$ and $q$, but depends on $\epsilon$.
\end{proposition}
\begin{proof}
    We use the first statement of Lemma \ref{lemma:convergence}, Proposition \ref{prop:Fn23-contribution}, Theorem \ref{thm:chebotarev} over a given $\mathrm{SL}_2(\mathbb{F}_l)$ Galois extension $K(E[l])/K$, and Theorem \ref{thm:SatheSelberg} to obtain
    \begin{align} \label{eqn:m0-part1}
    \begin{split}
        & \hspace{15pt} \sum_{f \in F_n^{main,0}(\mathbb{F}_q)} (\# \mathrm{Sel}_{1-\sigma_{l,f}}(A_f/K) - (l+1)) \\
        &= \sum_{k = R_2}^{R_1} \left( 1 - \left(\frac{l}{l^2-1} \right)^k \right) \cdot \frac{q^n}{n} \cdot \frac{(\log n)^{k-1}}{(k-1)!} \cdot \Biggl(G((k-1)/\log n) \cdot (\mathbb{E}[\beta(\delta_1,l,1)] - (l+1)) \cdot \gamma_l^{k-1} \\
        & \hspace{210pt} + O_{E,l} \left( \log n \cdot n^{\frac{3 \# \Sigma_E \log l}{\log \log \log n}} \cdot \gamma_l^{k-\omega(h)-1} \cdot q^{-1/2} \right) \\
        &\hspace{210pt} + O_{E,l} \left( \frac{k}{(\log n)^2} \cdot \gamma_l^{k-1} \right) \Biggl).
    \end{split}
    \end{align}
    Note that we use effective Chebotarev density theorem to quantify the proportion of irreducible factors of $h \in \mathcal{Q}(n)$ lying in $\mathcal{P}_0, \mathcal{P}_1$, and $\mathcal{P}_2$, which subsumes the exponent $\omega'(h)$ appearing in the statement of Lemma \ref{lemma:init-prob} as an exponent of $\gamma_l$. The resulting error term, which is of order $O(q^{-1/2})$, is subsumed into either one of the error terms, because $q > (\log n)^2$ by conditions on lower bounds for $n$ and $q$. We also note that all the other error terms appearing in Proposition \ref{prop:Fn23-contribution} and Theorem \ref{thm:chebotarev} are subsumed into two error terms. 

    The main term of equation (\ref{eqn:m0-part1}) can be written as
    \begin{align*}
        &= \sum_{k=R_2}^{R_1} \frac{q^n}{n} \cdot \frac{(\log n)^{k-1}}{(n-1)!} G((k-1)/\log n) \cdot (\mathbb{E}[\beta(\delta_1,l,1)] - (l+1)) \cdot \gamma_l^{k-1} \\
        &= \frac{q^n}{n} \cdot (\mathbb{E}[\beta(\delta_1,l,1)] - (l+1)) \cdot \left[ \sum_{k=R_2}^{R_1} \frac{(\gamma_l \log n)^{k-1}}{(k-1)!} \cdot G((k-1)/\log n) \right].
    \end{align*}
    Using Stirling's approximation, we obtain the following expression for any $\epsilon > 0$, where the implied constant depends on $\epsilon$.
    \begin{align*}
        &= \frac{q^n}{n} \cdot (\mathbb{E}[\beta(\delta_1,l,1)] - (l+1)) \cdot \left[ \sum_{k = \lfloor \gamma_l \log n - (\log n)^{\frac{1}{2} + \epsilon} \rfloor}^{\lceil \gamma_l \log n + (\log n)^{\frac{1}{2} + \epsilon} \rceil} \frac{(\gamma_l \log n)^{k-1}}{(k-1)!} \cdot G((k-1)/\log n) \right] \\
        &\hspace{20pt} \cdot (1 + O_\epsilon((\log n) \cdot \exp({-(\log n)^\epsilon}/2))).
    \end{align*}
    For any $c \in (0,e^{4 + 3 \log l})$ and any $1 \leq k \leq (\log n)^{1/2 + \epsilon}$ for any small enough $\epsilon > 0$, we have
    \begin{align}
        \begin{split}
            \frac{G \left(c + k/\log n \right)}{G(c)} &= \frac{\Gamma(c+1)}{\Gamma(c+1+k/\log n)} \cdot \prod_{g \in \mathcal{P}} \left(1 - \frac{1}{q^{\deg g}} \right)^{\frac{k}{\log n}} \cdot \left(1 + \frac{k/\log n}{q^{\deg g} + c} \right) \\
            &= 1 + O_{c,\epsilon}((\log n)^{-1/2+\epsilon}),
        \end{split}
    \end{align}
    with implied constant depending on $c$ and $\epsilon$ but can be chosen to be independent of $n$ and $q$. Combining the above expression with Stirling's approximation, we can rewrite the following expression as follows:
    \begin{align*}
        \begin{split}
            & \hspace{15pt} \sum_{k = \lfloor \gamma_l \log n - (\log n)^{\frac{1}{2} + \epsilon} \rfloor}^{\lceil \gamma_l \log n + (\log n)^{\frac{1}{2} + \epsilon} \rceil} \frac{(\gamma_l \log n)^{k-1}}{(k-1)!} \cdot G((k-1)/\log n) \\
            &= \sum_{k = \lfloor \gamma_l \log n - (\log n)^{\frac{1}{2} + \epsilon} \rfloor}^{\lceil \gamma_l \log n + (\log n)^{\frac{1}{2} + \epsilon} \rceil} \frac{(\gamma_l \log n)^{k-1}}{(k-1)!} \cdot G(\gamma_l) \cdot (1 + O_{\epsilon}((\log n)^{-1/2+\epsilon})) \\
            &= n^{\gamma_l} \cdot G(\gamma_l) \cdot (1 + O_{\epsilon}((\log n)^{-1/2+\epsilon})).
        \end{split}
    \end{align*}
    We note that the error terms from Taylor expansion of the exponential function gets subsumed into the error terms. Hence, the main term of equation (\ref{eqn:m0-part1}) can be rewritten as
    \begin{align*}
        &= \frac{q^n}{n} \cdot (\mathbb{E}[\beta(\delta_1,l,1)] - (l+1)) \cdot \left[ n^{\gamma_l} \cdot G(\gamma_l) \cdot (1 + O_\epsilon((\log n)^{-1/2 + \epsilon})) \right] \cdot (1 + O_\epsilon(e^{-\frac{1}{2} \cdot (\log n)^\epsilon})) \\
        &= q^n \cdot G(\gamma_l) \cdot (\mathbb{E}[\beta(\delta_1,l,1)] - (l+1)) \cdot n^{\gamma_l - 1} \cdot (1 + O_\epsilon((\log n)^{-1/2 + \epsilon})).
    \end{align*}

    By using the condition that $q > \frac{n^{\frac{8 \# \Sigma_E \log l - 2}{\log \log \log n} + 2 \epsilon}}{(\log n)^2}$ to ensure that the other error terms are strictly smaller than the main term, we obtain the expression
    \begin{equation}
        (\ref{eqn:m0-part1}) = q^n \cdot G(\gamma_l) \cdot (\mathbb{E}[\beta(\delta_1,l,1)] - (l+1)) \cdot n^{\gamma_l-1} \cdot (1 + O_{\epsilon}((\log n)^{-1/2 + \epsilon})).
    \end{equation}
    Using the fact that $\# F_n(\mathbb{F}_q) = q^n - q^{n-1} = q^n(1-1/q)$, we finish the proof of the proposition by identifying $\beta_{E,l} = G(\gamma_l) \cdot (\mathbb{E}[\beta(\delta_1,l,1)] - (l+1))$. Note that when we divide equation (\ref{eqn:m0-part1}) by $\# F_n(\mathbb{F}_q)$, we obtain the term $\frac{q}{q-1} = 1 + \frac{1}{q-1}$. The condition that $q > \frac{n^{\frac{8 \# \Sigma_E \log l - 2}{\log \log \log n} + 2 \epsilon}}{(\log n)^2}$ ensures that the term $\frac{1}{q-1}$ can be subsumed inside the error term $O_{\epsilon}((\log n)^{-1/2 + \epsilon})$.
\end{proof}
Next, we compute the contribution of rate of convergence over the set $F_n^{main,1}(\mathbb{F}_q)$.
\begin{proposition} \label{prop:convergence-m1}
    Assume conditions from Lemma \ref{lemma:convergence}. Suppose that $n$ and $q$ satisfy the same conditions as in Proposition \ref{prop:convergence-m0}. Then the following equation holds with explicitly computable implied constants.
    \begin{equation}
        \Biggl| \frac{\sum_{f \in F_n^{main,1}(\mathbb{F}_q)} (\# \mathrm{Sel}_{1-\sigma_{l,f}}(A_f/K) - (l+1))}{\# F_n(\mathbb{F}_q)} \Biggr| \asymp n^{\frac{l^3-l+1}{l^5} - 1}.
    \end{equation}
\end{proposition}
\begin{proof}
    We proceed as in the proof of Proposition \ref{prop:convergence-m1}. Recall the constant $\gamma_{l,1} = \frac{l^3 - l + 1}{l^4}$ from the statement of Lemma \ref{lemma:init-prob}. By the second equation in Lemma \ref{lemma:convergence}, we have
    \begin{align} \label{eqn:m1-part1}
    \begin{split}
        & \hspace{15pt} \sum_{f \in F_n^{main,1}(\mathbb{F}_q)} (\# \mathrm{Sel}_{1-\sigma_{l,f}}(A_f/K) - (l+1)) \\
        &= \sum_{k = R_2}^{R_1} \left( \frac{1}{l^k} - \left(\frac{1}{l^3-l} \right)^k \right) \cdot \frac{q^n}{n} \cdot \frac{(\log n)^{k-1}}{(k-1)!} \cdot \Biggl( G\left(\frac{k-1}{\log n} \right) \cdot (\mathbb{E}[\beta(\delta_1,l,1)] - (l+1)) \cdot \frac{\gamma_{l,1}^{k-1}}{l} \\
        & \hspace{210pt} + O_{E,l} \left( \log n \cdot n^{\frac{3 \# \Sigma_E \log l}{\log \log \log n}} \cdot \gamma_l^{k-\omega(h)-1} \cdot q^{-1/2} \right) \\
        &\hspace{210pt} + O_{E,l} \left( \frac{k}{(\log n)^2} \cdot \gamma_l^{k-1} \right) \Biggl).
    \end{split}
    \end{align}
    Unlike the proof of Proposition \ref{prop:convergence-m0}, however, we use effective Chebotarev density theorem to quantify the conditional proportion of irreducible factors of $h \in \mathcal{Q}(n)$, all of whose irreducible factors lie in either $\mathcal{P}_1$ or $\mathcal{P}_2$. The main contributing term originates from the first term. We can still use partial summation to obtain the statement of the proposition. In fact, one can adapt the proof of Proposition \ref{prop:convergence-m0} to obtain that the explicit coefficient for the secondary term is equal to $\frac{1}{l^2} \cdot G(\frac{l^3-l+1}{l^5}) \cdot (\mathbb{E}[\beta(\delta_1,l,1)] - (l+1))$.
\end{proof}
Lastly, we compute the contribution of rate of convergence over the set $F_n^{[1]}(\mathbb{F}_q)$.
\begin{proposition} \label{prop:convergence-[1]}
    Assume conditions from Lemma \ref{lemma:convergence}. Suppose that $n$ and $q$ satisfy the same conditions as in Proposition \ref{prop:convergence-m0}. Then the following inequality holds with explicitly computable implied constants.
    \begin{equation}
        \Biggl| \frac{\sum_{f \in F_n^{[1]}(\mathbb{F}_q)} \# \mathrm{Sel}_{1-\sigma_{l,f}}(A_f/K)}{\# F_n(\mathbb{F}_q)} \Biggr| \ll n^{1/(l^3-l) - 1}.
    \end{equation}
\end{proposition}
\begin{proof}
    We proceed as in the proof of Proposition \ref{prop:convergence-m0}, but use the last inequality in the statement of Lemma \ref{lemma:convergence}.
\end{proof}

We are now ready to state and prove the main result of the paper on determining the rate of convergence of dimensions of Selmer groups of prime cyclic twist families of some elliptic curves over $K$. Note that the simplified version of the main theorem as stated in the introduction can be obtained by choosing $l \geq 5$ and applying Remark \ref{rem:l5}.
\begin{theorem} \label{thm:main}
    Let $E$ be an elliptic curve satisfying Condition \ref{condition}. Recall the quantity $\mathbb{E}[\beta(\delta_1,l,1)]$ defined as (see Definition \ref{defn:initial-prob} and Remark \ref{rem:init-prob})
    \begin{equation}
        \mathbb{E}[\beta(\delta_1,l,1)] := \sum_{J=0}^\infty l^J \cdot \mathbb{P}[\delta_1 = J] = \lim_{n \to \infty} \sum_{J = 0}^\infty l^J \cdot \frac{\#\{f \in \mathcal{P}_0(n) : \dim_{\mathbb{F}_l} \mathrm{Sel}_{1-\sigma_{l,f}}(A_f/K) = J\}}{\# \mathcal{P}_0(n)}.
    \end{equation}
    If $l = 2, 3, $ or $5$, suppose that $E$ additionally satisfies 
    \begin{equation} \label{eq:main_thm_0}
        \mathbb{E}[\beta(\delta_1,l,1)] \neq l+1.
    \end{equation}
    Suppose that $n$ and $q$ satisfy the following conditions for some $\epsilon' > 0$:
    \begin{align} \label{eq:main_thm_1}
    \begin{split}
            n &> \exp(\exp(\exp(\exp({e + 2 \log l})))), \\
            q &> \max \left(\frac{n^{\frac{8 \# \Sigma_E \log l - 2}{\log \log \log n} + 2 \epsilon'}}{(\log n)^2}, (R_2 + \deg \Delta_E)^2, \exp({24 \log l + 4 \log D_{E,l} + (12 \log l) \cdot \# \Sigma_E}) \right).
    \end{split}
    \end{align}
    Recall the function $G(\cdot)$ from the analogue of Sathe-Selberg theorem (Theorem \ref{thm:SatheSelberg}). We set 
    \begin{equation} \label{eq:main_thm_2}
        \beta_{E,l} := G \left(1 - \frac{l^2-l+1}{l^3} \right) \cdot (\mathbb{E}[\beta(\delta_1,l,1)] - (l+1)) \neq 0.
    \end{equation}
    Choose $\epsilon > 0$ to be a small enough number. Then the following equation holds with implied constants depending on $E$, $l$, and $\epsilon$.
    \begin{equation}
        \frac{\sum_{f \in F_n(\mathbb{F}_q)} \# \mathrm{Sel}_{1-\sigma_{l,f}}(A_f/K) - (l+1)}{\# F_n(\mathbb{F}_q)} = \beta_{E,l} \cdot n^{-\frac{l^2 - l + 1}{l^3}} + O_{E,l,\epsilon} \left((\log n)^{-1/2 + \epsilon} \cdot n^{-\frac{l^2 - l + 1}{l^3}} \right).
    \end{equation}
\end{theorem}
\begin{proof}
    We have
    \begin{align*}
        & \frac{\sum_{f \in F_n(\mathbb{F}_q)} \# \mathrm{Sel}_{1-\sigma_{l,f}}(A_f/K) - (l+1)}{\# F_n(\mathbb{F}_q)} \\
        &= \sum_{j=0}^1 \frac{\sum_{f \in F_n^{main,j}(\mathbb{F}_q)} \# \mathrm{Sel}_{1-\sigma_{l,f}}(A_f/K) - (l+1)}{\# F_n(\mathbb{F}_q)} + \sum_{k=1}^3 \frac{\sum_{f \in F_n^{[k]}(\mathbb{F}_q)} \# \mathrm{Sel}_{1-\sigma_{l,f}}(A_f/K) - (l+1)}{\# F_n(\mathbb{F}_q)}.
    \end{align*}
    By Proposition \ref{prop:convergence-m0}, we have
    \begin{equation}
        \frac{\sum_{f \in F_n^{main,0}(\mathbb{F}_q)} \# \mathrm{Sel}_{1-\sigma_{l,f}}(A_f/K) - (l+1)}{\# F_n(\mathbb{F}_q)} = \beta_{E,l} \cdot n^{-\frac{l^2 - l + 1}{l^3}} + O_{E,l,\epsilon} \left((\log n)^{-1/2 + \epsilon} \cdot n^{-\frac{l^2 - l + 1}{l^3}} \right).
    \end{equation}
    By Proposition \ref{prop:convergence-m1}, we have (where the implied constants are positive)
    \begin{equation}
        \frac{\sum_{f \in F_n^{main,1}(\mathbb{F}_q)} \# \mathrm{Sel}_{1-\sigma_{l,f}}(A_f/K) - (l+1)}{\# F_n(\mathbb{F}_q)} \asymp n^{\frac{l^3-l+1}{l^5} - 1}.
    \end{equation}
    By Proposition \ref{prop:convergence-[1]}, we have (where the implied constants are positive)
    \begin{equation}
        \frac{\sum_{f \in F_n^{[1]}(\mathbb{F}_q)} \# \mathrm{Sel}_{1-\sigma_{l,f}}(A_f/K)}{\# F_n(\mathbb{F}_q)} \ll n^{1/(l^3-l) - 1}.
    \end{equation}
    By Proposition \ref{prop:Fn23-contribution}, we have (where the implied constants are positive)
    \begin{equation}
        \sum_{k=2}^3 \frac{\sum_{f \in F_n^{[k]}(\mathbb{F}_q)} \# \mathrm{Sel}_{1-\sigma_{l,f}}(A_f/K)}{\# F_n(\mathbb{F}_q)} \ll n^{-0.9} \cdot \sqrt{\log n}.
    \end{equation}
    Because $l \geq 2$, it follows that the main term for rate of convergence over the set $F_n^{main,0}(\mathbb{F}_q)$ dominates those over other sets. We note that if $l \geq 7$ then $\beta_{E,l} \neq 0$ for any elliptic curve $E$  satisfying Condition \ref{condition} by Remark \ref{rem:l5}.
\end{proof}

\begin{remark}
    It is a natural follow-up question to ask whether one can obtain secondary terms for higher moments of prime Selmer groups of cyclic prime twist families of $E$. As shown in Remark \ref{remark:generalization_d}, the recursive formula for expressing $d$-th moments of $M^k \delta$ using all moments up to $d$-th moments of $\delta$ further complicates the analysis of rates of convergence. Nevertheless, one may make an educated guess from equation (\ref{eqn:generalization_d}) that the secondary term for $d$-th moments of $\mathrm{Sel}_{1-\sigma_f}(A_f/K)$ may be also of order $O \left(n^{-\frac{l^2 - l + 1}{l^3}} \right)$. This is because the dominant geometric rate in equation (\ref{eqn:generalization_d}) would be equal to $\gamma_{l,1} = \gamma_l$ assuming that $$\kappa_1(k,l,d) \cdot \mathbb{E}[\beta(\delta,l,1)] - \prod_{m=1}^d(l^m+1) \neq 0, \hspace{20pt} \text{ for sufficiently large } k.$$ It would be an interesting question to explore whether the secondary terms for $d$-th moments of Selmer groups are indeed of order $O\left(n^{-\frac{l^2 - l + 1}{l^3}}\right)$.
\end{remark}

\begin{remark}
    Theorem \ref{thm:main} can be generalized to twist families of Selmer structures $\mathrm{Sel}(T, \alpha : \overline{\chi})$ given an alternating admissible Galois representation $T$, twisting data $\alpha$, and restriction of characters $\overline{\chi} \in \frac{\mathrm{Hom}(\mathrm{Gal}(\overline{K}/K), \mu_\ell)}{\mathrm{Aut}(\mu_\ell)}$ \cite[Section 2]{Park25-1}. We recall the definition of an alternating admissible Galois representation $T$ (see \cite[Condition 2.1]{Park25-1}):
    \begin{itemize}
        \item $T$ is a 2-dimensional $\mathbb{F}_l$ vector space.
        \item There exists a $\Gal(\overline{K}/K)$-equivariant alternating non-degenerate pairing $e_T: T \times T \to \mu_l$ which induces a symmetric pairing $H^1_{\et}(K_v, T) \times H^1_{\et}(K_v, T) \to H^2_{\et}(K_v, T \otimes T) \to \mathbb{F}_\ell$.
        \item The constant field of $K(T)$ is equal to $\mathbb{F}_q$.
        \item $T$ is a simple $\Gal(\overline{K}/K)$-module with $\mathrm{Hom}_{\Gal(\overline{K}/K)}(T,T) = \mathbb{F}_l$ and $H^1(\Gal(K(T)/K),T) = 0$.
        \item There exists a place $v$ of $K$ such that $T^{\Gal(\overline{K}_v/K_v\overline{\mathbb{F}}_q)} \cap \mu_{l^\infty} = \mu_l$.
    \end{itemize}
    Under these conditions, the proof of \cite[Theorem 2.11]{Park25-1} shows that there are explicit constants $p_0, p_1, p_2 \in [0,1]$ with $p_0 + p_1 + p_2 = 1$ such that the Markov operator
    \begin{equation}
        M_T := p_0 \cdot I + p_1 \cdot M_L + p_2 \cdot M_L^2
    \end{equation}
    is the governing stochastic process which determines the distribution of Selmer structures $\mathrm{Sel}(T, \alpha: \overline{\chi})$. Adapting the proof of Theorem \ref{thm:convergence}, we obtain for any probability distribution $\delta: \mathbb{Z}_{\geq 0} \to [0,1]$,
    \begin{equation}
        \mathbb{E}[\beta(M_T \delta, l, 1)] = \left( p_0 + p_1 \cdot \frac{1}{l} + p_2 \cdot \frac{1}{l^2} \right) \cdot \mathbb{E}[\beta(\delta,l,1)] + \left(  p_1 \cdot \frac{l^2-1}{l} + p_2 \cdot \frac{(l^2-1)(l+1)}{l^2} \right).
    \end{equation}
    Denote by $\gamma_l(p_1, p_2)$ and $\kappa_l(p_1,p_2)$ as follows. (Here we substitute $p_0 = 1 - p_1 - p_2$).
    \begin{equation}
        \gamma_l(p_1,p_2) := 1 - p_1 \cdot \frac{l-1}{l} - p_2 \cdot \frac{l^2-1}{l^2}, \hspace{15pt} \kappa_l(p_1,p_2) := p_1 \cdot \frac{l^2-1}{l} + p_2 \cdot \frac{(l^2-1)(l+1)}{l^2}.
    \end{equation}
    Then for any $k \geq 1$, we have
    \begin{equation}
        \mathbb{E}[\beta(M_T^k \delta, l, 1)] = \mathbb{E}[\beta(\delta,l,1)] \cdot \gamma_l(p_1, p_2)^k + \sum_{s=0}^{k-1} \kappa_l(p_1, p_2) \cdot \gamma_l(p_1, p_2)^{s}.
    \end{equation}
    After some calculation one can show that $\kappa_l(p_1,p_2) = (l+1) \cdot (1 - \gamma_l(p_1,p_2))$, hence regardless of values of $p_0, p_1, p_2$, the first moments of Selmer structures $\mathrm{Sel}(T, \alpha : \overline{\chi})$ converges to $l+1$, as one may deduce from \cite[Theorem 2.11]{Park25-1}.
    
    Take $\delta_1$ to be the distribution of Selmer structures
    \begin{equation}
        \delta_1(J) := \frac{\# \{\omega \in \Omega_1 : \dim_{\mathbb{F}_l} \mathrm{Sel}(T, \alpha : \overline{\chi}) = J\}}{\# \Omega_1}.
    \end{equation}
    Suppose that $(\mathbb{E}[\beta(\delta_1,l,1)] - (l+1)) \neq 0$.
    By adapting the strategy of the proof of Theorem \ref{thm:main}, and choosing $n$ and $q$ to be sufficiently large, we can obtain the following generalization:
    \begin{itemize}
        \item The secondary term for first moments of Selmer structures $\mathrm{Sel}(T, \alpha: \overline{\chi})$ over $F_n^{main,0}(\mathbb{F}_q)$ is 
        \begin{align*}
            & \hspace{15pt} G(\gamma_l(p_1,p_2)) \cdot (\mathbb{E}[\beta(\delta_1,l,1)] - (l+1)) \cdot n^{\gamma_{l}(p_1,p_2)-1} \\
            &= G \left(1 - p_1 \cdot \frac{l-1}{l} - p_2 \cdot \frac{l^2-1}{l^2} \right) \cdot (\mathbb{E}[\beta(\delta_1,l,1)] - (l+1)) \cdot n^{-p_1 \cdot \frac{l-1}{l} - p_2 \cdot \frac{l^2-1}{l^2}}.
        \end{align*}
        \item We define the constant
        $$\gamma_l^*(p_1,p_2) := p_1 \cdot \left(\frac{p_1}{p_1+p_2} \cdot \frac{1}{l} + \frac{p_2}{p_1+p_2} \cdot \frac{1}{l^2} \right).$$
        The secondary term for first moments of Selmer structures $\mathrm{Sel}(T, \alpha: \overline{\chi})$ over $F_n^{main,1}(\mathbb{F}_q)$ is $$\frac{p_1}{l} \cdot G \left( \gamma_l^*(p_1,p_2) \right) \cdot \left(\mathbb{E}[\beta(\delta_1,l,1)] - (l+1) \right) \cdot n^{\gamma_l^*(p_1,p_2)- 1}.$$
        \item The first moments of Selmer structures $\mathrm{Sel}(T, \alpha: \overline{\chi})$ over $F_n^{[1]}(\mathbb{F}_q)$ is of order $O(n^{p_2 - 1})$.
        \item The first moments of Selmer structures $\mathrm{Sel}(T, \alpha: \overline{\chi})$ over $F_n^{[2]}(\mathbb{F}_q) \cup F_n^{[3]}(\mathbb{F}_q)$ is of order $O(n^{-1 + \epsilon})$ for small enough $\epsilon > 0$.
    \end{itemize}
    The secondary term is thus determined by comparing the exponents of $n$ for the subsets $F_n^{main,0}(\mathbb{F}_q)$, $F_n^{main,1}(\mathbb{F}_q)$, and $F_n^{[1]}(\mathbb{F}_q)$. Given a small enough $\epsilon > 0$, we denote by $\mathscr{S}(n,p_1,p_2)$ the following expression:
    \begin{itemize}
        \item If $\gamma_l(p_1,p_2) > \max\{\gamma_l^*(p_1,p_2),p_2\}$, then $\mathscr{S}(n,p_1,p_2)$ is equal to
        \begin{equation*}
            G(\gamma_l(p_1,p_2)) \cdot (\mathbb{E}[\beta(\delta_1,l,1)] - (l+1)) \cdot n^{\gamma_l(p_1,p_2)-1}  + O_{T,l,\epsilon}(n^{\gamma_l(p_1,p_2)-1} \cdot (\log n)^{-1/2 + \epsilon}).
        \end{equation*}
        \item If $\gamma_l(p_1,p_2) = \gamma_l^*(p_1,p_2) > p_2$, then $\mathscr{S}(n,p_1,p_2)$ is equal to
        \begin{align*}
            \left( 1 + \frac{p_1}{l} \right) \cdot G(\gamma_l^*(p_1,p_2)) \cdot (\mathbb{E}[\beta(\delta_1,l,1)] - (l+1)) \cdot n^{\gamma_l(p_1,p_2)-1} + O_{T,l,\epsilon}(n^{\gamma_{l}(p_1,p_2) - 1} \cdot (\log n)^{-1/2 + \epsilon}).
        \end{align*}
        \item If $\gamma_l^*(p_1,p_2) < \max\{\gamma_l(p_1,p_2), p_2\}$, then $\mathscr{S}(n,p_1,p_2)$ is equal to
        \begin{equation*}
            \frac{p_1}{l} \cdot G \left( \gamma_l^*(p_1,p_2) \right) \cdot \left(\mathbb{E}[\beta(\delta_1,l,1)] - (l+1) \right) \cdot n^{\gamma_l^*(p_1,p_2)- 1} + O_{T,l,\epsilon}(n^{\gamma_l^*(p_1,p_2) - 1} \cdot (\log n)^{-1/2+\epsilon}).
        \end{equation*}
        \item Otherwise, we have $\mathscr{S}(n,p_1,p_2) = O_{T,l}(n^{p_2-1})$.
    \end{itemize}
    Then the following expression for first moments of $\mathrm{Sel}(T, \alpha : \overline{\chi_f})$ for sufficiently large $n$ and $q$ holds:
    \begin{equation}
        \frac{\sum_{f \in F_n(\mathbb{F}_q)} \# \mathrm{Sel}(T, \alpha : \overline{\chi}_f)}{\# F_n(\mathbb{F}_q)} = (l+1) + \mathscr{S}(n,p_1,p_2).
    \end{equation}
\end{remark}

\begin{remark}
    The recent groundbreaking work by Aaron Landesman and Ishan Levy \cite{LL25} states that the rate of convergence of moments of odd Selmer groups of quadratic twist families of abelian varieties (satisfying some mild conditions) to the Poonen-Rains distribution is of order $O(q^{\frac{-In + J}{2}})$ for some constants $I > 0$ and $J$, assuming that $q$ is chosen to be some sufficiently large but fixed prime power. What causes the difference between the logarithmic secondary term for the first moments of $2$-Selmer groups of quadratic twist families of elliptic curves obtained in this paper and the power saving error term obtained for the first moments of odd Selmer groups of those obtained in \cite{LL25}?

    The key difference lies in determining the image of the local Kummer map $\delta_v: \frac{E_f(K_v)}{mE_f(K_v)} \to H^1_{\et}(K_v, E[m])$ associated to multiplication by $m$ map. When $m$ is coprime to the Tamagawa number of $E$ at $v$, then we may use \cite[Proposition 5.4]{KCesnavicius} to conclude that the image of $\delta_v$ is the unramified cohomology group $H^1_{ur}(K_v, E[m])$. This is the key reason why previous works \cite{EL23, LL25} were able to utilize Hurwitz spaces - a geometric manifestation of parametrizing unramified Galois extensions over global function fields - to compute moments of odd Selmer groups of such quadratic twist families of $E$. One may then utilize stable homology results for such Hurwitz spaces to obtain the desired power saving rate of convergence.
    
    However, the implication from \cite[Proposition 5.4]{KCesnavicius} is no longer valid if $m$ is not coprime to the Tamagawa number of $E$ at $v$. Indeed, when $v$ has reduction type $I_0^*$, (i.e. the additive reduction of $E_f$ at place $v \mid f$ but $v \nmid \Delta_E$), the Tamagawa number of $E_f$ at $v$ is equal to $4$, so there is no guarantee that $\delta_v$ is the unramified cohomology group. Indeed, we observe from \cite{KMR14} that when $m = 2$, the image of $\delta_v$ is a totally ramified subspace of $H^1_{\et}(K_v, E_f[2])$. This prevents us from utilizing Hurwitz spaces in a similar manner presented in \cite{EL23, LL25}. (As a note, this subtlety is the reason why equation (3.20) appearing in \cite[Chapter 3, p.79]{Pa24} is incorrect, which makes the main theorem in the corresponding chapter to be regrettably incorrect; The Tamagawa number of $A_f$ at any place $v \mid f$ and $v \nmid \Delta_E$ is divisible by $l^2$. We sincerely thank Aaron Landesman and Will Sawin for pointing out this error.) Instead, we may use Lagrangian mod-$l$ Markov operators and their rates of convergence to obtain the claimed logarithmic saving secondary term. From this we observe that the secondary term for first moments of $p$-Selmer groups encapsulate differences in images of local Kummer maps between the case when $p = 2$ and the case when $p > 2$, in which the former case can be interpreted as a question in genus theory, whereas the latter case does not.
    
    There is an interesting geometric interpretation we can deduce from Theorem \ref{thm:main} (c.f. \cite[Problem 5, posed by Akshay Venkatesh]{GW21}). Let $\mathrm{Conf}_n(\mathbb{A}^1_{\mathbb{F}_q})$ be the unordered configuration space of $n$ points on $\mathbb{A}^1_{\mathbb{F}_q}$. Then one has $\mathrm{Conf}_n(\mathbb{A}^1_{\mathbb{F}_q})(\mathbb{F}_q) = F_n(\mathbb{F}_q)$. One may ask whether there is a sequence of geometric spaces $\{X_n\}_{n \geq 1}$ admitting a morphism $X_n \to \mathrm{Conf}_n(\mathbb{A}^1_{\mathbb{F}_q})$, such that the ratio $\frac{\#X_n(\mathbb{F}_q)}{\# F_n(\mathbb{F}_q)}$ recovers the first moments of $\mathrm{Sel}_{1-\sigma_{l,f}}(A_f/K)$ as $f$ varies over $F_n(\mathbb{F}_q)$. Suppose such a sequence of spaces exists. By comparing the rates of convergence obtained from Theorem \ref{thm:main} and \cite[Theorem 1.4.8]{LL25}, we may deduce that $X_n$ would not be a Hurwitz space $\mathrm{CHur}_{n}^{c,S}$ defined as in \cite[Section 2]{LL25}. This is because by Grothendieck-Lefschetz trace formula \cite[Theorem 25.1]{milneLEC}, Deligne's proof of the Weil conjectures \cite{Del74}, and Theorem \ref{thm:main}, we may deduce that the sequence of cohomology groups $\{H^i_\et((X_n)_{\overline{\mathbb{F}_q}}, \mathbb{Q}_l)\}_{n \geq 1}$ would not stabilize over a linear range of indices $i$.
\end{remark}

\nocite{*}
\bibliographystyle{alpha}
\bibliography{main}

\end{document}